\documentclass[11pt]{article}
\usepackage{amsmath, amssymb}
\usepackage{amsthm} %
\usepackage{enumerate}  %
 \usepackage{hyperref}
\usepackage[capitalize]{cleveref}
\usepackage{url} %
\usepackage{makeidx} %
\usepackage{tikz}
\usetikzlibrary{decorations.markings}
\usetikzlibrary{shapes.geometric, positioning}
\usepackage{array}
\usepackage{wrapfig}
\usepackage{nicefrac}
\usepackage{xspace}
\usepackage{fullpage}
\usepackage[dvipsnames]{xcolor}
\usepackage{graphicx}

\usepackage{pgfplots}
\pgfplotsset{compat=1.18}

\usepackage[nottoc]{tocbibind}

  \usepackage[numbers,sort&compress]{natbib}

\newtheoremstyle{plainsl}%
	{\topsep}
	{\topsep}
	{\slshape} %
	{}
	{\normalfont\bfseries}
	{.}
	{ }
	{}

{\theoremstyle{plainsl}
\newtheorem{theorem}{Theorem}[section]
\newtheorem{lemma}[theorem]{Lemma}
\newtheorem{corollary}[theorem]{Corollary}
\newtheorem{proposition}[theorem]{Proposition}

\newtheorem{openproblem}[theorem]{Open Problem}}
{\theoremstyle{remark}
\newtheorem{example}[theorem]{Example}}
\newtheorem{remark}[theorem]{Remark}
\usepackage{longtable}

\definecolor{colour1}{RGB}{230,97,0}
\definecolor{colour2}{RGB}{93,58,155}
\definecolor{colour3}{RGB}{0,102,102}

\newcommand\Zv{{\mathbf v}}

\begin{document}
\title{Sustainable trees}
\author{Ferenc Bencs\thanks{Centrum Wiskunde \& Informatica, Amsterdam, The Netherlands. Supported by the Netherlands Organisation of Scientific Research (NWO): VI.Veni.222.303., Email: \texttt{f.bencs@gmail.com}} \and Maarten van Dijk\thanks{Email: \texttt{m.z.j.vandijk00@gmail.com}} \and  Krystal Guo\thanks{Korteweg-de Vries
  Institute for Mathematics and QuSoft, University of Amsterdam, Amsterdam,
  The Netherlands. Email: \texttt{k.guo@uva.nl}} }
\date{October 12, 2026}
\maketitle

\begin{abstract}
 Let $X$ be a connected graph and let $\lambda_1(X)$ denote the spectral radius of its adjacency matrix. An edge $e$ of $X$ is \emph{sustainable} if it can be
  deleted and replaced by a non-edge so that the resulting graph is connected
  and
  its spectral radius does not decrease; $X$ is \emph{sustainable} if every edge
  is. We study this edgewise notion of spectral robustness for trees, where each
  replacement is an elementary exchange between trees on the same vertex set. We
  establish basic properties of sustainability and exhibit infinite families of
  sustainable and non-sustainable trees. Our main results concern trees with a
  vertex of large degree: if $T$ is a tree on $n$ vertices with
  $\Delta(T)\geq n/2$, then every vertex of maximum degree is a principal
  vertex,
  that is, a vertex at which the Perron vector is maximal. Building on this, we
  prove that every tree with $\Delta(T)\geq n/2+1$ is non-sustainable, and we
  classify the sustainable trees in the boundary case $\Delta(T)=n/2$: for
  $n\geq 4$, the only such trees are the double broom $B(2,n/2,n/2-1)$ and one
  exceptional tree on eight vertices. We conclude with open problems on the
  generalized tree-shift order and the asymptotic proportion of sustainable
  trees.


\end{abstract}

\section{Introduction}

Let $X$ be a connected graph, and let $\lambda_1(X)$ denote the spectral radius of its adjacency matrix. For an edge $e\in E(X)$, consider the operation of deleting $e$ and inserting a non-edge $f\in E(\overline X)$, subject to the requirement that the resulting graph remains connected. We say that the edge $e$ is\emph{sustainable} if at least one such replacement satisfies
\[
    \lambda_1(X-e+f)\geq \lambda_1(X),
\]
and we call $X$ \emph{sustainable} if every edge of $X$ is sustainable. Thus a non-sustainable graph contains a spectrally indispensable edge: an edge for which every admissible one-edge replacement strictly decreases the spectral radius.

This definition is intended as an edgewise notion of spectral robustness. Classical extremal questions ask which graphs in a prescribed class maximize or minimize the spectral radius. Sustainability asks a different, local question: starting from a given graph, can each individual edge be exchanged for another edge without making the spectral objective worse?

Trees provide a particularly natural setting for this question. If $T$ is a tree and $e\in E(T)$, then $T-e$ has exactly two components, and every admissible replacement joins a vertex of one component to a vertex of the other. Hence an edge replacement is precisely an elementary exchange between trees on the same vertex set. In this sense, sustainability studies the local geometry of the spectral radius on the exchange graph of trees.

The classical extremal trees already illustrate the two possible behaviours. Lov\'asz and Pelik\'an proved that, among trees of a fixed order, the path uniquely minimizes the adjacency spectral radius and the star uniquely maximizes it \cite{lovasz1973eigenvalues}; see also \cite{brouwer12,cvetkovic09}. It follows that $P_n$ is sustainable for $n\geq 3$: after deleting any edge, one can reconnect the two components to obtain another tree, whose spectral radius cannot be smaller than that of the path. On the other hand, $K_{1,n-1}$ is non-sustainable for $n\geq 4$, because replacing any leaf edge produces a non-isomorphic tree and therefore strictly decreases the spectral radius. Thus, sustainability is compatible with the familiar path–star extremal picture.

There is also a natural connection with Csikvári’s partial order on trees \cite{Csikvari2010}. This order is generated by generalized tree shifts, with the path as its minimum element and the star as its maximum element, and every nontrivial shift increases the spectral radius. Since these two extremal elements are respectively sustainable and non-sustainable, it is tempting to regard sustainability as measuring how far a tree has progressed from the path toward the star. The examples in \ref{fig:Non upperset}  show that this interpretation is only heuristic: sustainability is not monotone under any arbitrary generalized tree shifts.

The Perron vector reveals more precisely where an obstruction to sustainability can occur. By the Perron–Frobenius theorem, a connected graph has a unique positive unit eigenvector $\Zv$ associated with $\lambda_1$. We call a vertex \emph{principal} if its coordinate in $\Zv$ is maximal. The extreme entries and localization of principal eigenvectors have been studied extensively; see, for example,  \cite{PAPENDIECK2000129,cioabua2007principal,Das02012025}.
The edge-rotation lemma of \cite{Rowlinson} implies that every edge not incident with a principal vertex is sustainable: after deleting such an edge, the separated component can be reattached at a principal vertex, strictly increasing the spectral radius. Thus possible obstructions to sustainability are concentrated near the principal vertices. This observation leads naturally to the relation between the principal vertices and vertices of large degree, which is one of the main statement of the paper.

Our principal results are summarized below. The double broom $B(\ell,d_1,d_2)$ is defined in \cref{sec:preliminaries}, and $X_8$ denotes the tree obtained from disjoint copies of $K_{1,3}$ and $P_4$ by joining the center of $K_{1,3}$ to an internal vertex of $P_4$.

\begin{theorem}\label{thm:introduction-main-results}
Let $T$ be a tree on $n$ vertices.
\begin{enumerate}[(a)]
    \item If $\Delta(T)\geq n/2$, then every vertex of maximum degree is a principal vertex of $T$.
    \item If $\Delta(T)\geq n/2+1$, then $T$ is non-sustainable.
    \item Suppose that $n\geq 4$ and $\Delta(T)=n/2$. Then $T$ is sustainable if and only if
    \[
        T\cong B\!\left(2,\frac n2,\frac n2-1\right),
    \]
    or $n=8$ and $T\cong X_8$.
\end{enumerate}
\end{theorem}

The threshold in part~(b) is sharp: the family $B(2,n/2,n/2-1)$ is sustainable. Part~(c) shows that actually this is the only tree that could be a sustainable if $\Delta(T)=n/2$, apart from one additional tree of order eight. The proof requires a more delicate analysis because, after a leaf incident with a maximum-degree vertex is moved, the principal vertex of the new tree may change.

The paper is organized as follows. In \cref{sec:preliminaries} we collect the spectral tools used throughout, including edge rotations, edge contraction, and generalized tree shifts. In \cref{sec:sustainability} we establish basic properties and give several infinite families of sustainable and non-sustainable trees. We then study the location of principal vertices in trees with a high-degree vertex and prove part~(a) of \cref{thm:introduction-main-results}. The large-degree non-sustainability theorem and the classification in the boundary case are proved in the following section. We conclude with questions concerning the generalized tree-shift order, the asymptotic proportion of sustainable trees, and rooted attachment problems.

\section{Preliminaries}\label{sec:preliminaries}

\paragraph{Graphs and eigenvalues.}
We briefly recall the notation and spectral graph theory tools used throughout the paper; for background, see \cite{brouwer12,Godsil93}. Let $X$ be a finite simple graph, and let $A(X)$ be its adjacency matrix. If $X$ is connected, we denote the spectral radius of $A(X)$ by $\lambda_1(X)$. When the graph is clear from context, we write $A$ and $\lambda_1$.

For a vertex $u\in V(X)$, let $N(u)$ denote the neighbourhood of $u$. Additionally, let $N[u]=N(u)\cup\{u\}$ denote the closed neighbourhood of $u$. By the Perron--Frobenius theorem, if $X$ is connected, then $A(X)$ has a unique positive eigenvector corresponding to $\lambda_1(X)$. We call this vector the \emph{Perron vector} of $X$ and denote it by $\Zv$. A vertex $u\in V(X)$ is called a \emph{principal vertex} if
\[
    \Zv(u)=\max_{v\in V(X)}\Zv(v).
\]

For vertices $u,v\in V(X)$, we write $d(u,v)$ for their distance and $\deg(u)$ for the degree of $u$. Let $W_{uv}(X)$ denote the set of all walks in $X$ starting at $u$ and ending at $v$; when $u=v$, we write $W_u(X)$. Finally, let
\[
    \phi(X,t)=\det(tI-A(X))
\]
denote the characteristic polynomial of $X$.
\begin{proposition}\cite[Proposition 5.1.1]{brouwer12}\label{prop:charpoly-tree-recurrence}
Let $T$ be a tree and let $u\in V(T)$. Then
\[
    \phi(T,t)= t\phi(T-u,t) - \sum_{v\in N(u)}\phi(T-\{u,v\},t).
\]
\end{proposition}

Denote the set of all paths from $u$ to $v$ in $X$ by $\mathcal{P}_{uv}(X)$. If $P\in \mathcal{P}_{uv}(X)$, then $X\setminus P$ denotes the graph obtained from $X$ by deleting all vertices of $P$.

\begin{proposition}\cite[Section 4.2]{Godsil93}\label{prop:walk-generating-function}
Let $X$ be a graph and let $u,v\in V(X)$. For every $t>\lambda_1(X)$,
\[
    \sum_{P\in \mathcal{P}_{uv}(X)}
    \frac{\phi(X\setminus P,t)}{\phi(X,t)}
    =
    t^{-1}\sum_{\gamma\in W_{uv}(X)} t^{-|\gamma|},
\]
where $|\gamma|$ denotes the length of the walk $\gamma$.
\end{proposition}

Let $X$ be a graph with adjacency matrix $A$. For $u,v\in V(X)$, define $\phi_{uv}(X,t)$ by
\[
    \phi_{uv}(X,t) = \operatorname{adj}(tI-A)_{uv}.
\]
Equivalently,
\[
    \frac{\phi_{uv}(X,t)}{\phi(X,t)}     =     (tI-A)^{-1}_{uv}
\]
whenever $tI-A$ is invertible. When $u=v$, we write $\phi_u(X,t)$ for $\phi_{uu}(X,t)$.

\begin{lemma}\cite[Lemma~4.2.1]{Godsil93}\label{lem: cofactor characterstic}
    Let $u$ and $v$ be vertices of a graph $X$. Then
    \[
        \phi_{uv}(X,t)=\sum_{P\in \mathcal{P}_{uv}(X)}\phi(X\setminus P, t).
    \]
\end{lemma}

The following is \cite[Lemma~4.2.3]{Godsil93} with the added hypothesis that $\phi_{11}(X,\theta)\neq 0$, which is implicitly assumed in Godsil's proof.

\begin{lemma}\cite[Lemma~4.2.3]{Godsil93}\label{lem: vector relation to characterstic polynomial}
    Let $X$ be a  graph on $n$ vertices with adjacency matrix $A$ and let $\theta$ be an eigenvalue with multiplicity 1. Let $\Zv$ be the vector with $i$-th entry equal to the value $\phi_{1i}(X,\theta)$. If $\phi_{11}(X,\theta)\neq 0$, then $\Zv\neq 0$ and $A\Zv=\theta\Zv$.

    In particular, if $X$ is connected and $\theta=\lambda_1(X)$, then the corresponding $\Zv$ is the Perron vector.
\end{lemma}
\begin{lemma}\cite[Christoffel-Darboux identity, Lemma~4.3.3]{Godsil93}\label{lem:christoffel-darboux}
    Let $X$ be a graph and let $u$ and $v$ be vertices of $X$. Then
    \[\phi_{uv}(X,t)\phi'(X,t)-\phi_{uv}'(X,t)\phi(X,t)=\sum_{w\in V(X)}\phi_{uw}(X,t)\phi_{vw}(X,t). \]
\end{lemma}

We define a \emph{double broom} as follows. Let $\ell\geq 1$ and let $d_1,d_2\geq 1$. The double broom $B(\ell,d_1,d_2)$ is obtained from a path of length $\ell$ with end vertices $u$ and $v$ by attaching $d_1-1$ leaves to $u$ and $d_2-1$ leaves to $v$. Thus
\[
    \deg(u)=d_1
    \qquad\text{and}\qquad
    \deg(v)=d_2.
\]
We call $u$ and $v$ the centers of the double broom. The double broom is called \emph{balanced} if $d_1=d_2$. For example, the tree $B(n-4,3,3)$ is the well known $\widehat D_n$, see \cref{fig:Dynkin graph}.

\paragraph{Increasing the spectral radius of graphs}
In this paragraph we present several tools that will be used to compare the spectral radii of trees. The first is the classical monotonicity of the spectral radius.
\begin{theorem}[Proposition~3.1.1,\cite{brouwer12}]\label{thm:spectral-radius-monotonicity}
Let $X$ be a connected graph and let $Y$ be a proper subgraph of $X$. Then $\lambda_1(X)$ is a simple eigenvalue and
\[
    \lambda_1(Y)<\lambda_1(X).
\]
\end{theorem}

Another useful tool for increasing the spectral radius of a graph is the following \emph{edge rotation} operation.
\begin{lemma}[Theorem 3.2, \cite{Rowlinson}]\label{lem: ERL}
    Let $X$ be a connected simple graph and suppose that $u, v, w\in V(X)$ such that $e=vw\in E(X)$ and $u$ is not adjacent to $w$.
    Let $\Zv$ be the Perron vector of $X$  and assume that $\Zv(u)\geq \Zv(v)$. Let $X'=X-vw+uw$, then
    \[\lambda_1(X')>\lambda_1(X). \]
\end{lemma}
For our purposes, we will need a more general version of \cref{lem: ERL} which we give below.

\begin{lemma}\label{lem:general-ERL}
    Let $X$ be a connected simple graph on $n$ vertices, and $u, v_1,\dots,v_k\in V(X)$ for some $k\ge 1$, where $u \neq v_i$ but the $v_i$ are not necessarily distinct. For each $i$, let $w_i$ be a neighbour of $v_i$ such that $u\neq w_i$, $u$ is not adjacent to $w_i$ and $w_i\neq w_j$ whenever $i\neq j$. Additionally that the edges $v_iw_i$, $i=1,\dots,k$, are pairwise distinct. Let $\Zv$ be the Perron vector of $X$ and assume for all $i$ that  \[\Zv(u)\geq\Zv(v_i).\] Let $X'$ be the graph obtained from $X$ by replacing the edges $v_iw_i$ by $uw_i$ for all $i=1,\dots,k$. Then
    \[\lambda_1(X')>\lambda_1(X).\]
\end{lemma}
\begin{proof} Let $B = A(X') - A(X)$; we can write out the entries of B as follows:
  \[B(a, b)=\begin{cases}
        1, & \text{if $\{a,b\} = \{u, w_i\}$ for some $i$;}\\
        -1, & \text{if $\{a,b\} = \{v_i,w_i\}$  for some $i$;}\\
        0, & \text{otherwise.}
    \end{cases}\]
    Let $\Zv$ be the unit Perron vector to $X$, then
    \[\lambda_1(X')\geq \Zv^TA(X')\Zv=\Zv^TA(X)\Zv+\Zv^TB\Zv. \]
    By definition of $B$ we have
    \[\Zv^TB\Zv=2\left(\sum_{i=1}^k\Zv(u)\Zv(w_i)-\Zv(v_i)\Zv(w_i)\right).\]
    Since $\Zv(u)\geq\Zv(v_i)$ for every $i$, we have $\Zv^TB\Zv\geq0$, thus
    \begin{equation}\label{eq:edge-rotation-main}
        \lambda_1(X')\geq \Zv^TA(X')\Zv=\Zv^TA(X)\Zv+\Zv^TB\Zv\geq\Zv^TA(X)\Zv=\lambda_1(X).
    \end{equation}
   It remains to rule out equality. Suppose, for contradiction, that $\lambda_1(X)=\lambda_1(X')$, then by the previous inequalities we have that $\Zv$ is an eigenvector for $X$ and $X'$. However, $u$ has strictly more neighbours in $X'$ than in $X$, thus
    \[
        \lambda_1(X)\Zv(u)=\sum_{uv\in E(X')}\Zv(v)>\sum_{uv\in E(X)}\Zv(v)=\lambda_1(X)\Zv(u),
    \]
    which is a contradiction.
\end{proof}

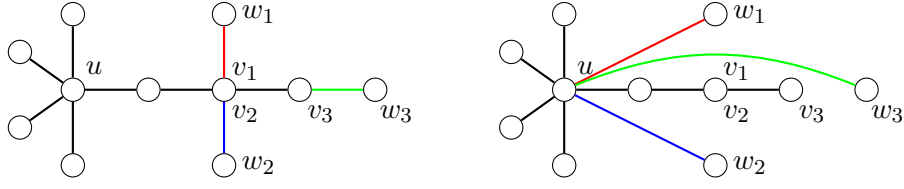
\begin{figure}[htb]
    \centering
    \begin{tikzpicture}[main_node/.style={circle,draw,minimum size=0.8em,inner sep=3pt}]

        \node[main_node] (0) at (0,1) {};
        \node[main_node] (1) at (0,-1) {};
        \node[main_node] (2) at (-0.7,-0.5) {};
        \node[main_node] (3) at (-0.7, 0.5) {};
        \node[main_node] (4) at (0,0) {};
        \node[above, yshift=2pt, xshift=8pt] at (4) {$u$};
        \node[main_node] (5) at (1,0) {};
        \node[main_node] (6) at (2,0) {};
        \node[above, xshift=8pt] at (6) {$v_1$};
        \node[below, yshift=-2pt, xshift=8pt] at (6) {$v_2$};
        \node[main_node] (7) at (3,0) {};
        \node[below, yshift=-2pt, xshift=8pt] at (7) {$v_3$};
        \node[main_node] (8) at (2,1) {};
        \node[xshift=13pt] at (8) {$w_1$};
        \node[main_node] (9) at (2,-1) {};
        \node[xshift=13pt] at (9) {$w_2$};
        \node[main_node] (10) at (4,0) {};
        \node[below, yshift=-2pt, xshift=8pt] at (10) {$w_3$};
        \path[draw, thick]
        (4) edge node {} (3)
        (4) edge node {} (2)
        (4) edge node {} (1)
        (4) edge node {} (0)
        (4) edge node {} (5)
        (5) edge node {} (6)
        (6) edge[red] node {} (8)
        (6) edge[blue] node {} (9)
        (6) edge node {} (7)
        (7) edge[green] node {} (10);
        \begin{scope}[xshift=6.5cm]
            \node[main_node] (11) at (-0.7,0.5) {};
            \node[main_node] (12) at (0,0) {};
            \node[above, yshift=2pt, xshift=8pt] at (12) {$u$};
            \node[main_node] (13) at (-0.7,-0.5) {};
            \node[main_node] (14) at (0,1) {};
            \node[main_node] (15) at (0,-1) {};
            \node[main_node] (16) at (1,0) {};
            \node[main_node] (17) at (2,0) {};
            \node[above, xshift=8pt] at (17) {$v_1$};
            \node[below, yshift=-2pt, xshift=8pt] at (17) {$v_2$};
            \node[main_node] (18) at (2,1) {};
            \node[xshift=13pt] at (18) {$w_1$};
            \node[main_node] (19) at (3,0) {};
            \node[below, yshift=-2pt, xshift=8pt] at (19) {$v_3$};
            \node[main_node] (20) at (4,0) {};
            \node[below, yshift=-2pt, xshift=8pt] at (20) {$w_3$};
            \node[main_node] (21) at (2,-1) {};
            \node[xshift=13pt] at (21) {$w_2$};

            \path[draw, thick]
            (12) edge node {} (14)
            (12) edge node {} (13)
            (12) edge node {} (11)
            (12) edge node {} (15)
            (12) edge node {} (16)
            (16) edge node {} (17)
            (18) edge[red] node {} (12)
            (21) edge[blue] node {} (12)
            (19) edge node {} (17);
            \draw[thick, green] (20) to [bend right=22] (12);
        \end{scope}

    \end{tikzpicture}
    \caption{A tree illustrating the application of \cref{lem:general-ERL}. Note that in this example, we have $v_1=v_2$, but $w_1\neq w_2$.}
    \label{fig:example of general ERL}
\end{figure}

The next technique is by edge contraction. Edge contraction is a natural graph operation and the effect on the spectral radius is well understood. A path $P$ in a graph $X$ is called pendent if all internal vertices of $P$ have degree $2$ in $X$ and one endpoint of $P$ has degree $1$ in $X$.
\begin{proposition}\cite{mckee2018contraction}\label{prop:contraction}
    Let $X$ be a graph and let $e\in E(X)$.  If $\lambda_1(X)\ge 2$ and $e$ is not on any pendent path, then
    \[
    \lambda_1(X/e)>\lambda_1(X),
    \]
    and if $e$ is on a pendent path, then
    \[
        \lambda_1(X/e)<\lambda_1(X).
    \]
\end{proposition}
A direct corollary is the following lemma that we will use later, is what we call \emph{edge teleportation (ET)} operation.
\begin{lemma}\label{lem: contracting plus adding leaf}
    Let $T$ be a tree with $\lambda_1(T)\geq 2$. Let $v\in V(T)$ and $e\in E(T)$ given, such that $e$ is not on a pendent path of $T$. Let $T'$ be the tree obtained from $T$ by adding a new leaf at $v$ and then contracting $e$. Then
    \[
        \lambda_1(T)<\lambda_1(T').
    \]
\end{lemma}
\begin{proof}
Let $T_1$ be the tree obtained from $T$ by attaching a new leaf at $v$. By \cref{thm:spectral-radius-monotonicity},
\[
    \lambda_1(T_1)>\lambda_1(T)\geq 2.
\]

We claim that $e$ is not on any pendent path of $T_1$. Indeed, if such a pendent path in $T_1$ did not use the new leaf, then it would already be a pendent path in $T$. If it did use the new leaf, then deleting the new leaf from that path would give a pendent path in $T$ containing $e$. In either case, this contradicts the assumption that $e$ is not on any pendent path of $T$.

Thus, by \cref{prop:contraction}, we have
\[
    \lambda_1(T')>\lambda_1(T_1)>\lambda_1(T).\qedhere
\]
\end{proof}
An example of the edge teleportation can be seen in \cref{fig: example edge teleportation}.
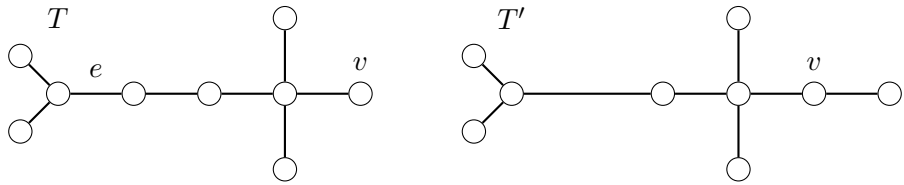
\begin{figure}[htb]
    \centering
    \begin{tikzpicture}[main_node/.style={circle,draw,minimum size=0.8em,inner sep=3pt}]
        \begin{scope}
            \node[main_node] (0) at (-0.5,0.5) {};
            \node[main_node] (1) at (0,0) {};
            \node[main_node] (2) at (-0.5, -0.5) {};
            \node[main_node] (3) at (1,0) {};
            \node[main_node] (4) at (2,0) {};
            \node[main_node] (5) at (3,0) {};
            \node[main_node] (6) at (4,0) {};
            \node[main_node] (7) at (3,1) {};
            \node[main_node] (8) at (3 , -1) {};
            \node[] at (0.5,0.3) {$e$};
            \node[] at (4,0.4) {$v$};
            \node[] at (0,1) {$T$};

             \path[draw, thick]
            (1) edge node {} (0)
            (1) edge node {} (2)
            (1) edge node {} (3)
            (3) edge node {} (4)
            (4) edge node {} (5)
            (5) edge node {} (6)
            (5) edge node {} (7)
            (5) edge node {} (8)
            ;
        \end{scope}
        \begin{scope}[xshift=6cm]
            \node[main_node] (0) at (-0.5,0.5) {};
            \node[main_node] (1) at (0,0) {};
            \node[main_node] (2) at (-0.5, -0.5) {};
            \node[main_node] (4) at (2,0) {};
            \node[main_node] (5) at (3,0) {};
            \node[main_node] (6) at (5,0) {};
            \node[main_node] (7) at (3,1) {};
            \node[main_node] (8) at (3 , -1) {};
            \node[main_node] (9) at (4 , 0) {};
            \node[] at (0,1) {$T'$};
            \node[] at (4,0.4) {$v$};

             \path[draw, thick]
            (1) edge node {} (0)
            (1) edge node {} (2)
            (1) edge node {} (4)
            (4) edge node {} (5)
            (5) edge node {} (8)
            (5) edge node {} (7)
            (5) edge node {} (9)
            (6) edge node {} (9)
            ;
        \end{scope}
        \end{tikzpicture}
    \caption{A tree $T$ and a tree $T'$ obtained from $T$ by applying edge teleportation to the edge $e$ and the vertex $v$.}
    \label{fig: example edge teleportation}
\end{figure}

The final operation that we introduce is the \emph{generalized tree shift (GTS)}, first defined in \cite[Definition~2.1]{Csikvari2010}. Let $T$ be a tree, and let $u,v\in V(T)$ be vertices such that all internal vertices of the path from $u$ to $v$, if any, have degree $2$ in $T$. Let $w$ be the neighbour of $v$ lying on the path from $u$ to $v$. The \emph{generalized tree shift} of $T$ from $v$ to $u$ is the tree $T'$ obtained by deleting all edges between $v$ and the vertices in $N(v)\setminus\{w\}$ and adding all edges between $u$ and the vertices in $N(v)\setminus\{w\}$.

We call $u$ the \emph{beneficiary} and $v$ the \emph{candidate} of the generalized tree shift. See \cref{fig:GTS} for a schematic example.

\begin{figure}[htb]
    \centering
    \begin{tikzpicture}[main_node/.style={circle,draw,minimum size=0.8em,inner sep=3pt]},scale=0.9]

            \begin{scope}   \node[main_node] (0) at (-2, 1.5) {};
            \node[main_node] (1) at (-2,0.7) {};
            \node[main_node] (2) at (-1,1) {};
            \node[main_node] (3) at (-2,0 ) {};
            \node[main_node] (4) at (-1,0) {};
            \node[main_node] (5) at (-1,-0.8) {};
            \node[main_node] (6) at (0,0) {};
            \node[below, yshift=-4pt] at (6) {$u$};
            \node[main_node] (7) at (1,0) {};
            \node[main_node] (8) at (2,0) {};
            \node[main_node] (9) at (3,0) {};
            \node[below, yshift=-4pt] at (9) {$v$};
            \node[main_node] (10) at (4,0.5) {};
            \node[main_node] (11) at (5,1) {};
            \node[main_node] (12) at (5,0) {};
            \node[main_node] (13) at (5.5,0.5) {};
            \node[main_node] (14) at (5,-0.8) {};
            \node[main_node] (15) at (4,-0.8) {};

            \draw (-1.6,0.3) ellipse (1.2cm and 2cm);
            \draw (4.6,0.3) ellipse (1.2cm and 2cm);

            \path[draw, thick]
            (0) edge (2)
            (1) edge (2)
            (3) edge (4)
            (5) edge (6)
            (6) edge (4)
            (6) edge (2)
            (6) edge (7)
            (7) edge (8)
            (8) edge (9)
            (9) edge (10)
            (9) edge (15)
            (10) edge (11)
            (10) edge (12)
            (12) edge (14)
            (12) edge (13);

        \end{scope}

        \begin{scope}[shift={(10,0)}]

            \node[main_node] (17) at (-2,1.5) {};
            \node[main_node] (18) at (-1,1) {};
            \node[main_node] (19) at (-2,0.7) {};
            \node[main_node] (20) at (-2,0) {};
            \node[main_node] (21) at (-1,0) {};
            \node[main_node] (22) at (0,0) {};
            \node[below, yshift=-4pt] at (22) {$u$};
            \node[main_node] (23) at (-1,-0.8) {};
            \node[main_node] (24) at (0,1) {};
            \node[main_node] (25) at (0,2) {};
            \node[main_node] (26) at (0,3) {};
            \node[right, xshift=2pt] at (26) {$v$};
            \node[main_node] (27) at (2,0.5) {};
            \node[main_node] (28) at (3,1) {};
            \node[main_node] (29) at (3,0) {};
            \node[main_node] (30) at (3.5,0.5) {};
            \node[main_node] (31) at (3,-0.8) {};
            \node[main_node] (32) at (2,-0.8) {};

            \draw (-1.6,0.3) ellipse (1.2cm and 2cm);
            \draw (2.6,0.3) ellipse (1.2cm and 2cm);

            \path[draw, thick]
            (17) edge (18)
            (18) edge (19)
            (20) edge (21)
            (21) edge (22)
            (22) edge (23)
            (22) edge (18)
            (22) edge (24)
            (24) edge (25)
            (25) edge (26)
            (22) edge (27)
            (27) edge (28)
            (27) edge (29)
            (29) edge (30)
            (29) edge (31)
            (32) edge (22);

        \end{scope}
    \end{tikzpicture}
    \caption{The generalized tree shift demonstrated with $u$ as the beneficiary and $v$ as the candidate.}
    \label{fig:GTS}
\end{figure}
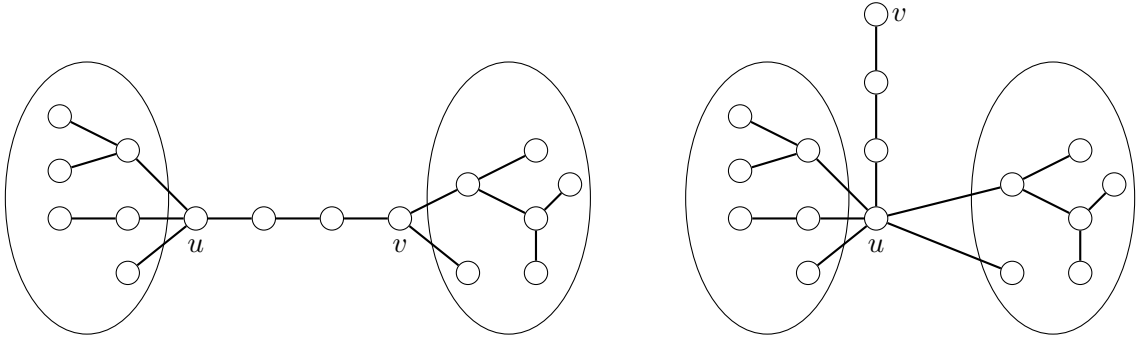
\begin{theorem}[Corollary 4.7, \cite{Csikvari2010}]\label{thm: GTS increasing radius}
    Let $T'$ be a generalized tree shift of $T$. Then,
    \[\lambda_1(T')\geq\lambda_1(T).\]
    Equality holds if and only if $T'\cong T$; otherwise, the inequality is strict.
\end{theorem}

\section{Sustainability}\label{sec:sustainability}
Recall that an edge of a tree is sustainable if it can be removed and replaced by another edge so that the resulting graph is again a tree and the spectral radius does not decrease. Since the path $P_n$ uniquely minimizes the spectral radius among trees on $n$ vertices, every edge of $P_n$ is sustainable. Similarly, since the star uniquely maximizes the spectral radius among trees on $n$ vertices for $n\ge 4$, every star on at least 4 vertices is non-sustainable.

These two extremal examples suggest that sustainability is related to how path-like or star-like a tree is. In this section we prove several basic properties of sustainable trees and give families of sustainable and non-sustainable examples.
\subsection{Basic properties}

\begin{lemma}\label{lem:ERLsustainable}
    Suppose $T$ is a tree and $u$ is a principal vertex of $T$. If $e=vw$ is an edge not incident to $u$, then there exists $x\in\{v,w\}$ such that the graph
    \[T' = T - e + ux\]
    is a tree and $\lambda_1(T') > \lambda_1(T)$.
\end{lemma}
\begin{proof}
    Since $T$ is a tree, deleting $e$ results in a forest with exactly two connected components, say $T_1$ and $T_2$, where  $v\in V(T_1)$ and $w\in V(T_2)$.  If $u\in V(T_1)$, then adding the edge $uw$ reconnects the two components and results in a tree. If $u\in V(T_2)$ then adding the edge $uv$ does so. In either case, there exists a unique $x\in\{v,w\}$ such that the graph
    \[T' = T - e + ux\]
    is a tree.
    Let $\Zv$ be the Perron vector of $T$. Since $u$ is a principal vertex, we have $\Zv(u)\geq\Zv(v)$ and $\Zv(u)\geq\Zv(w)$. In particular $\Zv(u)\ge\Zv(x) $. thus by \cref{lem: ERL} we have
    \[\lambda_1(T')>\lambda_1(T). \qedhere \]
\end{proof}
Thus, to determine whether a tree is sustainable, it remains only to check the edges incident to all principal vertices. Thus, the question whether a tree is sustainable, reduces to verifying if the edges incident to the principal vertices are sustainable.

We obtain the following as immediate consequences.

\begin{corollary}\label{cor: multiple principal vertices is sustainable}
    Let $T$ be a tree and $U$ be the set of principal vertices.
    \begin{enumerate}[(a)]
        \item If $U=\{u\}$ and every edge $uw$, where $w\in N(u)$, is sustainable, then $T$ is sustainable.
        \item If $U=\{u,v\}$ and $uv$ is a sustainable edge, then $T$ is sustainable.
        \item If $U=\{u,v\}$ and $u,v$ are not adjacent,  then $T$ is sustainable.
        \item If $|U| \geq 3$,  then $T$ is sustainable.
    \end{enumerate}
\end{corollary}
\begin{proof}
   By \cref{lem:ERLsustainable}, if there exists a principal vertex $x$ not incident to $e$, then replacing $e$ with an edge incident to an endpoint of $e$ and $x$ results in a graph with higher spectral radius. All four statements follow directly.
\end{proof}

\begin{lemma}\label{lem: incident degree 2 is sustainable}
    Let $T$ be a tree and suppose that $e$ is an edge that is incident to a vertex of degree 2, then $e$ is sustainable.
\end{lemma}
\begin{proof}
    Let $u$ be a vertex of degree 2 with neighbours $v$ and $w$, and let $e=uv$.  Let $u$ be the candidate and $w$ the beneficiary of the generalized tree shift. This operation replaces the edge $uv$ by $vw$, producing a tree $T'$. By Theorem~\ref{thm: GTS increasing radius}, $T'$ has as large spectral radius as $T$. Thus $e$ is sustainable.
\end{proof}

\subsection{Examples of sustainable trees}
As mentioned before, the path graph on $n$ vertices, $P_n$, is sustainable. In the examples below, we give two other families of trees which are sustainable.

\begin{example}
    The graph $\widehat{D}_n$ for $n\geq 6$ is seen in \cref{fig:Dynkin graph}, note that the number of vertices is one more than the index given. The Perron vector of this graph is given by
    \[
    \Zv(u)=\begin{cases}
        1, & \text{if $u$ is a leaf;}\\
        2, & \text{otherwise.}
    \end{cases}
    \]

    For $n\geq 6$, the tree has at least three non-leaf vertices and therefore at least three principal vertices. By \cref{cor: multiple principal vertices is sustainable}, the tree is sustainable.

    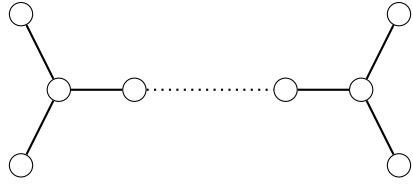
\begin{figure}[htb]
        \centering
        \begin{tikzpicture}[main_node/.style={circle,draw,minimum size=0.8em,inner sep=3pt}]

            \node[main_node] (0) at (-0.5, 1) {};
            \node[main_node] (1) at (-0.5, -1) {};
            \node[main_node] (2) at (0,0) {};
            \node[main_node] (3) at (1,0) {};
            \node[main_node] (4) at (3,0) {};
            \node[main_node] (5) at (4,0) {};
            \node[main_node] (6) at (4.5, 1) {};
            \node[main_node] (7) at (4.5, -1) {};

             \path[draw, thick]
            (0) edge node {} (2)
            (2) edge node {} (1)
            (2) edge node {} (3)
            (3) edge[dotted] node {} (4)
            (4) edge node {} (5)
            (5) edge node {} (6)
            (5) edge node {} (7)
            ;

        \end{tikzpicture}
        \caption{The graph $\widehat{D}_n$, where the number of vertices is $n+1$.}
        \label{fig:Dynkin graph}
    \end{figure}
\end{example}

\begin{example}\label{ex: starlike sustainable with n1 at least 2}
    Let $k\geq 3$ be an integer and $n_1\leq\dots\leq n_k$ be strictly positive integers. A starlike tree, denoted $S(n_1,\dots,n_k)$, is a tree with exactly one vertex of degree $k$, say $u$, such that
    \[
    S(n_1,\dots,n_k)-u\cong P_{n_1}\cup\dots\cup P_{n_k}.
    \]
    An example of a starlike tree is shown in \cref{fig:Starlike tree}. If $n_i\geq 2$ for all $i$, then every edge is incident to a vertex of degree two; hence, by \cref{lem: incident degree 2 is sustainable}, the tree is sustainable.
    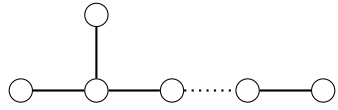
\begin{figure}[h!]
        \centering
        \begin{tikzpicture}[main_node/.style={circle,draw,minimum size=0.8em,inner sep=3pt]}]
            \node[main_node] (0) at (0,0) {};
            \node[main_node] (1) at (0,1) {};
            \node[main_node] (2) at (-1,0) {};
            \node[main_node] (3) at (1,0) {};
            \node[main_node] (4) at (2,0) {};
            \node[main_node] (5) at (3,0) {};

             \path[draw, thick]
            (0) edge node {} (1)
            (0) edge node {} (2)
            (0) edge node {} (3)
            (3) edge[dotted] node {} (4)
            (4) edge node {} (5)
            ;

            \end{tikzpicture}
        \caption{The starlike tree $S(1,1,k)$, also known as $D_{k+3}$ for $k\geq 2$.}
        \label{fig:Starlike tree}
    \end{figure}

\end{example}

\subsection{Examples of non-sustainable trees}
As mentioned before the star graphs are non-sustainable on at least 4 vertices. In this section, we give other families of trees which are non-sustainable.
\begin{example}\label{ex: double star non-sustainable}
    Let $k,m\geq 2$ be integers. The double star graph, denoted $S_{k,m}$ is a tree obtained by adding an edge between the degree $k$ and degree $m$ vertex of the graphs $S_k$ and $S_m$ and we say it is balanced if $k=m$. Equivalently, it can be viewed as the double broom $B(1,k+1,m+1)$. A graphical example is shown in \cref{fig:Double star tree}. The edge between the two stars is non-sustainable. In \cref{fig:non-sustainable edge double star}, all possible replacements of this edge are shown up to isomorphism. Applying the generalized tree shift one or more times yields the double star graph $S_{k,m}$. By \cref{thm: GTS increasing radius}, it follows that all trees in \cref{fig:non-sustainable edge double star} have strictly smaller spectral radius. Consequently, $S_{k,m}$ is non-sustainable.
    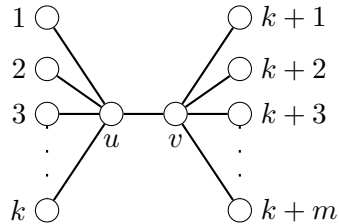
\begin{figure}
        \centering
        \begin{tikzpicture}[main_node/.style={circle,draw,minimum size=0.8em,inner sep=3pt]},scale=0.85]

        \node[main_node] (0) at (-1,1.5) {};
        \node[left, xshift=-4pt] at (0) {$1$};

        \node[main_node] (1) at (-1,0.7) {};
        \node[left, xshift=-4pt] at (1) {$2$};

        \node[main_node] (2) at (-1,0) {};
        \node[left, xshift=-4pt] at (2) {$3$};

        \node[main_node] (3) at (-1,-1.5) {};
        \node[left, xshift=-4pt] at (3) {$k$};

        \node[main_node] (4) at (0,0) {};
        \node[below, yshift=-4pt] at (4) {$u$};

        \node[main_node] (5) at (1,0) {};
        \node[below, yshift=-4pt] at (5) {$v$};

        \node[main_node] (6) at (2,1.5) {};
        \node[right, xshift=4pt] at (6) {$k+1$};
        \node[main_node] (7) at (2,0.7) {};
        \node[right, xshift=4pt] at (7) {$k+2$};
        \node[main_node] (8) at (2,0) {};
        \node[right, xshift=4pt] at (8) {$k+3$};
        \node[main_node] (9) at (2,-1.5) {};
        \node[right, xshift=4pt] at (9) {$k+m$};

         \path[draw, thick]
        (0) edge node {} (4)
        (1) edge node {} (4)
        (2) edge node {} (4)
        (4) edge node {} (5)
        (5) edge node {} (6)
        (5) edge node {} (7)
        (5) edge node {} (8)
        (4) edge node {} (3)
        (5) edge node {} (9)
        ;
        \path[draw, thick, dash pattern=on 1pt off 8pt]
        (2) edge node {} (3)
        (8) edge node {} (9);
        \end{tikzpicture}
        \caption{The double star graph $S_{k,m}$}
        \label{fig:Double star tree}
    \end{figure}

    \begin{figure}[htb]
        \centering
        \begin{tikzpicture}[main_node/.style={circle,draw,minimum size=0.8em,inner sep=3pt},scale=0.85]

        \begin{scope}
            \node[main_node] (0) at (-1,1.5) {};
            \node[left, xshift=-4pt] at (0) {$1$};
            \node[main_node] (1) at (-1,0.7) {};
            \node[left, xshift=-4pt] at (1) {$2$};
            \node[main_node] (2) at (-1,0) {};
            \node[left, xshift=-4pt] at (2) {$3$};
            \node[main_node] (3) at (-1,-2) {};
            \node[left, xshift=-4pt] at (3) {$k$};
            \node[main_node] (4) at (0,0) {};
            \node[below, yshift=-4pt] at (4) {$u$};
            \node[main_node] (5) at (1,0) {};
            \node[below, yshift=-4pt] at (5) {$v$};
            \node[main_node] (6) at (2,1.5) {};
            \node[right, xshift=4pt] at (6) {$k+1$};
            \node[main_node] (7) at (2,0.7) {};
            \node[right, xshift=4pt] at (7) {$k+2$};
            \node[main_node] (8) at (2,0) {};
            \node[right, xshift=4pt] at (8) {$k+3$};
            \node[main_node] (9) at (2,-2) {};
            \node[right, xshift=4pt] at (9) {$k+m$};

            \path[draw, thick]
            (0) edge (4)
            (1) edge (4)
            (2) edge (4)
            (4) edge (6)
            (5) edge (6)
            (5) edge (7)
            (5) edge (8)
            (4) edge (3)
            (5) edge (9);
        \end{scope}

        \begin{scope}[xshift=6cm]
            \node[main_node] (10) at (-1,1.5) {};
            \node[left, xshift=-4pt] at (10) {$1$};
            \node[main_node] (11) at (-1,0.7) {};
            \node[left, xshift=-4pt] at (11) {$2$};
            \node[main_node] (12) at (-1,0) {};
            \node[left, xshift=-4pt] at (12) {$3$};
            \node[main_node] (13) at (-1,-2) {};
            \node[left, xshift=-4pt] at (13) {$k$};
            \node[main_node] (14) at (0,0) {};
            \node[below, yshift=-4pt] at (14) {$u$};
            \node[main_node] (15) at (1,0) {};
            \node[below, yshift=-4pt] at (15) {$v$};
            \node[main_node] (16) at (2,1.5) {};
            \node[right, xshift=4pt] at (16) {$k+1$};
            \node[main_node] (17) at (2,0.7) {};
            \node[right, xshift=4pt] at (17) {$k+2$};
            \node[main_node] (18) at (2,0) {};
            \node[right, xshift=4pt] at (18) {$k+3$};
            \node[main_node] (19) at (2,-2) {};
            \node[right, xshift=4pt] at (19) {$k+m$};

            \path[draw, thick]
            (10) edge (14)
            (11) edge (14)
            (12) edge (14)
            (10) edge (15)
            (15) edge (16)
            (15) edge (17)
            (15) edge (18)
            (14) edge (13)
            (15) edge (19);
        \end{scope}

        \begin{scope}[xshift=12cm]
            \node[main_node] (20) at (-1,1.5) {};
            \node[left, xshift=-4pt] at (20) {$1$};
            \node[main_node] (21) at (-1,0.7) {};
            \node[left, xshift=-4pt] at (21) {$2$};
            \node[main_node] (22) at (-1,0) {};
            \node[left, xshift=-4pt] at (22) {$3$};
            \node[main_node] (23) at (-1,-2) {};
            \node[left, xshift=-4pt] at (23) {$k$};
            \node[main_node] (24) at (0,0) {};
            \node[below, yshift=-4pt] at (24) {$u$};
            \node[main_node] (25) at (1,0) {};
            \node[below, yshift=-4pt] at (25) {$v$};
            \node[main_node] (26) at (2,1.5) {};
            \node[right, xshift=4pt] at (26) {$k+1$};
            \node[main_node] (27) at (2,0.7) {};
            \node[right, xshift=4pt] at (27) {$k+2$};
            \node[main_node] (28) at (2,0) {};
            \node[right, xshift=4pt] at (28) {$k+3$};
            \node[main_node] (29) at (2,-2) {};
            \node[right, xshift=4pt] at (29) {$k+4$};

            \path[draw, thick]
            (20) edge (24)
            (21) edge (24)
            (22) edge (24)
            (20) edge (26)
            (25) edge (26)
            (25) edge (27)
            (25) edge (28)
            (24) edge (23)
            (25) edge (29);
        \end{scope}

        \path[draw, thick, dash pattern=on 1pt off 8pt]
        (2) edge (3)
        (8) edge (9)

        (12) edge (13)
        (18) edge (19)

        (22) edge (23)
        (28) edge (29)
        ;

        \end{tikzpicture}
        \caption{All the representations to replace the edge between the two star graphs up to isomorphism.}
        \label{fig:non-sustainable edge double star}
    \end{figure}
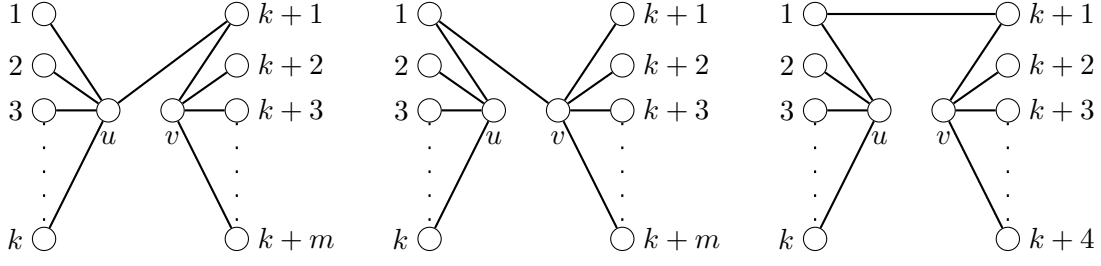

\end{example}

The following lemma will help us show that another family of trees is non-sustainable.
\begin{lemma}\cite[Lemma 3.5.4]{Stev2018}\label{lemma: equal path lengths}
    Let $u$ be a vertex of a nontrivial connected graph $X$ and, for nonnegative integers $k$ and $\ell$, let $X(u; \ell, k)$ denote the graph obtained from $X$ by attaching pendent paths of lengths $\ell$ and $k$ at $u$. If $k\geq \ell\geq 1$, then
    \[\lambda_1(X(u; \ell,k))>\lambda_1(X(u; \ell-1,k+1)).\]
\end{lemma}
If we take $X = P_2$, then $X(u; \ell,k) = S(1, \ell, k)$. The lemma implies that redistributing length by shortening the shorter pendent path and extending the longer one by one, i.e.,  $S(1, \ell-1, k+1)$, strictly decreases the spectral radius. Thus, for fixed $n_2+n_3$, the spectral radius of $S(1,n_2,n_3)$ is maximized when the two lengths are as equal as possible; equivalently, when $|n_2-n_3|\leq1$. With this, we can give the full characterization of sustainability of the family $S(n_1,\dots,n_k)$. Note that in \cref{ex: starlike sustainable with n1 at least 2}, we have shown that if $n_i\geq 2$ for all $i$, then $S(n_1,\dots,n_k)$ is sustainable. We now proceed to resolve the case when $n_1=1$.
\begin{proposition}\label{prop: starlike trees}
    Let $1\le n_2\le n_3$. Then the graph $S(1,n_2,n_3)$ is sustainable if and only if $n_2\neq n_3$. More generally, for $k \geq 4$, every tree in the family $S(1,n_2,\dots,n_k)$ is non-sustainable.
\end{proposition}
\begin{proof}
    Let $u$ be the vertex of degree 3, and let $v$ be a leaf adjacent to $u$. We first show that the edge $uv$ is non-sustainable if and only if $n_2=n_3$. The tree $S(1, n_2, n_3)$ can be obtained by taking $X=P_2$ and $\ell=n_2$ and $k=n_3$ in \cref{lemma: equal path lengths}. Replacing $uv$ by $vw$, with $w\neq u$ and $\deg(w)\neq 1$, results in $S(1, n_2', n_3')$. If $n_2=n_3$, then we have $n_2'<n_2$ and $n_3'>n_3$ and hence,
    \cref{lemma: equal path lengths} implies
    \[\lambda_1(S(1, n_2, n_3))>\lambda_1(S(1, n_2', n_3')).\]
    If $\deg(w)=1$, then we obtain $P_{n+1}$ and we know $\lambda_1(P_\ell)<\lambda_1(T)$ for any tree $T\ncong P_\ell$ on $\ell$ vertices. Thus the edge $uv$ and by definition, the tree $S(1,n_2,n_3)$ is non-sustainable if $n_2=n_3$. We will now show that if $n_2\neq n_3$, then $S(1, n_2, n_3)$ is sustainable. Let $u$ and $v$ as before. If $n_2 < n_3-1$, the edge $uv$ is sustainable since \cref{lemma: equal path lengths} implies
    \[
    \lambda_1(S(1,n_2',n_3')) > \lambda_1(S(1,n_2,n_3)),
    \]
    where $n_2'=n_2+1$ and $n_3'=n_3-1$.
    If $n_2 = n_3-1$, then $S(1, n_2+1,n_3-1) \cong S(1,n_2,n_3)$, thus the edge $uv$ is again sustainable. It remains to consider the other edges of $S(1, n_2, n_3)$ when $n_2\neq n_3$. If $n_2,n_3>1$ then all other edges of $S(1,n_2,n_3)$ are incident to a vertex of degree 2, and are therefore sustainable by \cref{lem: incident degree 2 is sustainable}. If $n_2=1$ and $n_3\neq n_2$, let $w$ be the other leaf adjacent to $u$. We have already showed that the edge $uv$ is sustainable, and the remaining edge incident to $u$ and a leaf is sustainable by symmetry. Thus, $S(1,n_2, n_3)$ is sustainable if and only if $n_2\neq n_3$.

    Let $k\geq 4, T=S(1,n_2,\dots, n_k)$ and let $u\in V(T)$ be of degree $k$ and $v$ be a leaf adjacent to $u$. To show that it is non-sustainable we show that $\lambda(T)>\lambda(T')$ for any $T'$, where $T'$ is obtained from $T$ by replacing the edge $uv$ with $vw$ for some $w\neq u$.
    Without loss of generality we may assume that $w$ is not a leaf of $T$. Indeed, if $w$ is attached to the $i$th pendent path, then applying a generalized tree shift on the edge on this path that is incident to $u$ we  obtain a tree isomorphic to $T$. This means that $\lambda(T')< \lambda(T)$ by \cref{thm: GTS increasing radius}.

    Now assume that $w$ is not a leaf of $T$ and let $\ell$ be the leaf, such that $w$ is on the path between $u$ and $\ell$. Also let $P$ be the path between $u$ and $w$ in $T'$. Observe that, by contracting all the edges of $P$ in $T'$ and attaching a path of the length $e(P)$ at $\ell$ we obtain a tree isomorphic to $T$. Thus by \cref{prop:contraction} we have that $\lambda(T')<\lambda(T)$.
\end{proof}

\begin{proposition}\label{prop: double star with pendant is non-sustainable}
    Let $k,m\geq 3$, and let $T$ be obtained from the double broom graph $B(1,k+1,m+1)$ by attaching a new leaf to one of the leaves adjacent to one of the two centers. Then $T$ is non-sustainable.
\end{proposition}
\begin{proof}
Let $u$ and $v$ be the two centers of the underlying double broom $B(1,k+1,m+1)$, and suppose that the extra leaf is attached to a leaf adjacent to $u$. Let $w$ be the unique vertex of degree $2$, and let $\ell_w$ be the leaf adjacent to $w$. We show that the edge $uv$ is not sustainable.

Delete the edge $uv$. Let $T'_1$ be the component containing $u$ and let $T'_2$ be the component containing $v$. For $x\in V(T'_1)$ and $y\in V(T'_2)$, let
\[
    T_{xy}=T-uv+xy.
\]
Thus every possible replacement of the edge $uv$ is of this form.

Using generalized tree shifts inside the two components, every tree $T_{xy}$ can be transformed, with nondecreasing spectral radius, into either $T$ itself or into the tree $T_{wv}$, see \cref{fig: replacing middle edge double star with a leaf} for the schematic transformations of these trees. More precisely, if $x$ is a leaf adjacent to $u$, a generalized tree shift with beneficiary $u$ moves the replacement edge from $x$ to $u$; if $x=\ell_w$, a generalized tree shift with beneficiary $w$ moves the replacement edge from $\ell_w$ to $w$. Similarly, if $y$ is a leaf adjacent to $v$, a generalized tree shift with beneficiary $v$ moves the replacement edge from $y$ to $v$. Therefore, by \cref{thm: GTS increasing radius}, it remains only to compare $T_{wv}$ with $T$.

Starting from $T_{wv}$, attach a new leaf at one of the leaves adjacent to $u$, and then contract the edge $uw$. The resulting tree is isomorphic to $T$. In $T_{wv}$, the vertices $u$ and $w$ both have degree at least $3$, so the edge $uw$ is not on a pendent path. Moreover, $T_{wv}$ contains a double broom of the form $B(\ell,3,3)$ as a subgraph, and hence $\lambda_1(T_{wv})\geq 2$. By \cref{lem: contracting plus adding leaf},
\[
    \lambda_1(T)>\lambda_1(T_{wv}).
\]
It follows that every possible replacement of $uv$ has spectral radius strictly smaller than $\lambda_1(T)$. Hence $uv$ is not sustainable, and therefore $T$ is non-sustainable.
\end{proof}

\begin{figure}[htb]
    \centering
    \begin{tikzpicture}[main_node/.style={circle,draw,minimum size=0.8em,inner sep=3pt},scale=0.8]
    \begin{scope}
        \node[main_node] (0) at (-1,1.5) {};
        \node[above, yshift=4pt] at (0) {$w$};
        \node[main_node] (10) at (-2,1.5) {};
        \node[above, yshift=4pt] at (10) {$\ell_w$};

        \node[main_node] (1) at (-1,0.7) {};
        \node[left, xshift=-4pt] at (1) {$\ell_u$};

        \node[main_node] (2) at (-1,0) {};

        \node[main_node] (3) at (-1,-1.5) {};

        \node[main_node] (4) at (0,0) {};
        \node[below, yshift=-4pt] at (4) {$u$};

        \node[main_node] (5) at (1,0) {};
        \node[below, yshift=-4pt] at (5) {$v$};

        \node[main_node] (6) at (2,1.5) {};
        \node[above, yshift=4pt] at (6) {$\ell_v$};

        \node[main_node] (7) at (2,0.7) {};

        \node[main_node] (8) at (2,0) {};

        \node[main_node] (9) at (2,-1.5) {};

        \node at (-3, 2.5) {$T$};

         \path[draw, thick]
        (0) edge node {} (4)
        (1) edge node {} (4)
        (2) edge node {} (4)
        (4) edge node {} (5)
        (5) edge node {} (6)
        (5) edge node {} (7)
        (5) edge node {} (8)
        (4) edge node {} (3)
        (5) edge node {} (9)
        (10) edge node {} (0)
        ;
        \path[draw, thick, dash pattern=on 1pt off 8pt]
        (2) edge node {} (3)
        (8) edge node {} (9);

    \end{scope}
    \begin{scope}[xshift=8cm, yshift=0.5cm]
        \node at (0,2) {$T$};
        \node at (-3,0) {$T_{wv}$};
        \node at (3,0) {$T_{u\ell_v}$};
        \node at (3,-2) {$T_{\ell_u\ell_v}$};
        \node at (0,0) {$T_{\ell_uv}$};
        \node at (-3,-2) {$T_{\ell_wv}$};

        \node at (-1,-2) {$T_{w\ell_v}$};
        \node at (1,-2) {$T_{\ell_w\ell_v}$};

    \draw[->] (-2.5, 0) -- (-0.5, 1.8 ) node[midway, sloped, yshift=5pt] {\tiny ET};

    \draw[->] (0,0.5) -- (0,1.5 ) node[midway, sloped, yshift=5pt] {\tiny GTS};

    \draw[->] (2.5, 0) -- (0.5,1.8 ) node[midway, sloped, yshift=5pt] {\tiny GTS};

    \draw[->] (-3,-1.5) -- (-3, -0.5 ) node[midway, sloped, yshift=5pt] {\tiny GTS};

    \draw[->] (-1.4,-1.5) -- (-2.5, -0.3 ) node[midway, sloped, yshift=5pt] {\tiny GTS};

    \draw[->] (3,-1.5) -- (3, -0.5 ) node[midway, sloped, yshift=5pt] {\tiny GTS};

    \draw[->] (0.3,-2) -- (-0.5, -2 ) node[midway, sloped, yshift=5pt] {\tiny GTS};

    \end{scope}

    \end{tikzpicture}
    \caption{The tree $T$ and a schematic illustration of how the different $T_{xy}$ can be transformed into one another.}
    \label{fig: replacing middle edge double star with a leaf}
\end{figure}
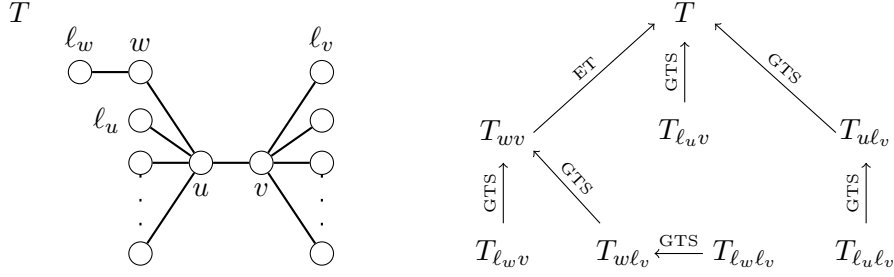

\section{Perron vector of high-degree trees}\label{sec:high_max_degree}

In this section we prove that principal vertices lie close to vertices of large degree. More precisely, if a tree $T$ on $n$ vertices contains a vertex of degree at least $n/2$, then every vertex of maximum degree is a principal vertex. The proof uses characteristic-polynomial identities for trees, together with Perron--Frobenius theory.

\begin{theorem}\label{thm: distance to perron at most k}
    Let $T$ be a tree on $n$ vertices and $u\in V(T)$. Let $k\geq 0$ be an integer. If $\deg(u) \geq \frac{n-k+1}{2}$, then $d(u,v) \leq k$ for any principal vertex $v$.

    Moreover, if $k\geq 1$, then equality $d(u,v)=k$ occurs if and only if $T$ is a balanced double broom with centers $u$ and $v$; equivalently, \[ T\cong B(k,d,d) \] for some $d$.
\end{theorem}
Before proving this theorem, we collect some useful lemmas. First, we bound the ratio of the entries of Perron vector $\Zv$ of a graph.
\begin{lemma}\label{lemma:ratio bound}
    Let $X$ be a connected graph and let $u,v$ be two vertices. Let $Y$ be a connected subgraph of $X$ such that $u,v\in V(Y)$. Then
    \[
        \frac{\Zv(v)}{\Zv(u)} \le \frac{\phi_{v}(Y,\lambda_1)}{\phi_{uv}(Y,\lambda_1)}=\frac{\phi(Y-v,\lambda_1)}{\sum_{P\in\mathcal{P}_{uv}}\phi(Y-P,\lambda_1)},
    \]
    where $\lambda_1=\lambda_1(X)$.

    Moreover, if we denote by $S$ the set  $\{v\}$ union the set of vertices of the connected component of  $X-v$ containing $u$, then we have equality if and only if the induced subgraph $Z=X[S]$ is a subgraph of $Y$.
\end{lemma}

\begin{proof}
Since $X$ is connected, $\lambda_1$ is a simple eigenvalue. Hence, by \cref{lem: vector relation to characterstic polynomial}, we may choose the eigenvector $\Zv$ (up to
scaling) so that
\[
\Zv(w)=\phi_{vw}(X,\lambda_1)
\qquad\text{for all }w\in V(X).
\]
In particular,
\[
\frac{\Zv(v)}{\Zv(u)}
=
\frac{\phi_v(X,\lambda_1)}{\phi_{uv}(X,\lambda_1)}.
\]
Now \cref{lem: cofactor characterstic} gives
\[
\frac{\phi_v(X,\lambda_1)}{\phi_{uv}(X,\lambda_1)}
=
\frac{\phi(X-v,\lambda_1)}
     {\displaystyle\sum_{P\in\mathcal P_{uv}(X)}\phi(X-P,\lambda_1)}
=
\frac{1}
     {\displaystyle\sum_{P\in\mathcal P_{uv}(X)}
      \frac{\phi(X-P,\lambda_1)}{\phi(X-v,\lambda_1)}}.
\]
Therefore it is enough to prove that
\[
\sum_{P\in\mathcal P_{uv}(X)}
\frac{\phi(X-P,\lambda_1)}{\phi(X-v,\lambda_1)}
\ge
\sum_{P\in\mathcal P_{uv}(Y)}
\frac{\phi(Y-P,\lambda_1)}{\phi(Y-v,\lambda_1)}.
\]

Since $\mathcal P_{uv}(Y)\subseteq \mathcal P_{uv}(X)$, and every summand is positive by \cref{thm:spectral-radius-monotonicity}, it suffices to show that for each $P\in\mathcal P_{uv}(Y)$,
\[
\frac{\phi(X-P,\lambda_1)}{\phi(X-v,\lambda_1)}
\ge
\frac{\phi(Y-P,\lambda_1)}{\phi(Y-v,\lambda_1)}.
\]

Fix $P=(v_0,\dots,v_\ell)\in\mathcal P_{uv}(Y)$ with $v_0=v$ and $v_\ell=u$.
For $0\le k\le \ell$, write $P_k=(v_0,\dots,v_k)$. Then
\[
\frac{\phi(X-P,\lambda_1)}{\phi(X-v,\lambda_1)}
=
\prod_{k=1}^{\ell}
\frac{\phi(X-P_k,\lambda_1)}{\phi(X-P_{k-1},\lambda_1)}
=
\prod_{k=1}^{\ell}
\frac{\phi\left((X-P_{k-1})-v_k,\lambda_1\right)}{\phi(X-P_{k-1},\lambda_1)}.
\]
Using \cref{prop:walk-generating-function} with $\lambda=\lambda_1$, $H=X-P_{k-1}$, and $x=v_k$, we obtain
\[
\frac{\phi\left((X-P_{k-1})-v_k,\lambda_1\right)}{\phi(X-P_{k-1},\lambda_1)}
=
\lambda_1^{-1}
\sum_{\gamma\in W_{v_k}(X-P_{k-1})}\lambda_1^{-|\gamma|},
\]
and similarly
\[
    \frac{\phi\left((Y-P_{k-1})-v_k,\lambda_1\right)}{\phi(Y-P_{k-1},\lambda_1)}
    =\lambda_1^{-1}\sum_{\gamma\in W_{v_k}(Y-P_{k-1})}\lambda_1^{-|\gamma|}.
\]
Since $Y-P_{k-1}$ is a subgraph of $X-P_{k-1}$, we have
\[
W_{v_k}(Y-P_{k-1})\subseteq W_{v_k}(X-P_{k-1}),
\]
and hence
\[
    \frac{\phi\left((X-P_{k-1})-v_k,\lambda_1\right)}{\phi(X-P_{k-1},\lambda_1)}
    \ge
    \frac{\phi\left((Y-P_{k-1})-v_k,\lambda_1\right)}{\phi(Y-P_{k-1},\lambda_1)}.
\]
Multiplying these inequalities over $k=1,\dots,\ell$ yields
\[
    \frac{\phi(X-P,\lambda_1)}{\phi(X-v,\lambda_1)}
    \ge
    \frac{\phi(Y-P,\lambda_1)}{\phi(Y-v,\lambda_1)}.
\]
Therefore,
\[
    \frac{\Zv(v)}{\Zv(u)}
    =
    \frac{1}{\sum_{P\in\mathcal P_{uv}(X)}\frac{\phi(X-P,\lambda_1)}{\phi(X-v,\lambda_1)}}
    \le
    \frac{1}{\sum_{P\in\mathcal P_{uv}(Y)}\frac{\phi(Y-P,\lambda_1)}{\phi(Y-v,\lambda_1)}}
    =
    \frac{\phi_v(Y,\lambda_1)}{\phi_{uv}(Y,\lambda_1)}.
\]

For the equality case, let $C$ be the connected component of $X-v$ containing $u$, let
\[
    S=V(C)\cup\{v\}, \qquad Z=X[S].
\]
All summands in the preceding inequalities are positive. Hence equality holds if and only if no additional $u,v$ path occurs in $X$ and every walk comparison used above is an equality. If $Z$ is a subgraph of $Y$, then every $u,v$ path in $X$ lies in $Z$, and after the initial vertex $v$ has been deleted, every walk starting at a vertex on such a path and remaining in the component containing $u$ also lies in $Z$. Thus every comparison above is an equality.
Conversely, suppose that $Z$ is not a subgraph of $Y$. Since $C$ is connected, either $X$ contains an $u,v$ path not contained in $Y$, which contributes an additional positive summand, or for some $P=(v_0,\dots,v_\ell)\in\mathcal P_{uv}(Y)$ and some $k\geq1$ there is a closed walk from $v_k$ in $X-P_{k-1}$ using an edge or vertex of $Z$ not present in $Y$. In the latter case the corresponding walk-generating-function inequality is strict.
Hence equality holds if and only if $Z$ is a subgraph of $Y$.
\end{proof}

\begin{lemma}\label{lemma: monotonicity of ratio}
    For any connected graph $Y$ and $u\neq v\in V(Y)$ we have
    \[
        t \mapsto\frac{\phi_v(Y,t)}{\phi_{uv}(Y,t)}
    \]
    is strictly monotone increasing on $\left(\lambda_1(Y-v),\infty\right)$.
\end{lemma}
\begin{proof}
    As before, let us use the fact that
    \[
        \frac{\phi_v(Y,t)}{\phi_{uv}(Y,t)}=\frac{1}{\sum_{P\in\mathcal P_{uv}(Y)}\frac{\phi(Y-P,t)}{\phi(Y-v,t)}},
    \]
    thus it is sufficient to show that for any $P\in\mathcal{P}_{uv}(Y)$ the map $\frac{\phi(Y-P,t)}{\phi(Y-v,t)}$ is positive and strictly monotone decreasing on $(\lambda_1(Y-v),\infty)$.

    Fix $P=(v_0,\dots,v_\ell)\in\mathcal P_{uv}(Y)$ with $v_0=v$ and $v_\ell=u$. By interlacing we have that
    \[
        \phi(Y-P,t), \phi(Y-v,t)>0
    \]
    for any $t> \lambda_1(Y-v)$. To show the monotonicity for $0\le k\le \ell$, write $P_k=(v_0,\dots,v_k)$. Then
    \[
        \frac{\phi(Y-P,t)}{\phi(Y-v,t)}
        =
        \prod_{k=1}^{\ell}
        \frac{\phi(Y-P_k,t)}{\phi(Y-P_{k-1},t)}
        =
        \prod_{k=1}^{\ell}
        \frac{\phi\left((Y-P_{k-1})-v_k,t\right)}{\phi(Y-P_{k-1},t)}.
    \]
    By the Christoffel-Darboux identity, \cref{lem:christoffel-darboux}, we have that for any graph $H$ with vertex $v$ we have
    \[
        \left(\frac{\phi(H-v,t)}{\phi(H,t)}\right)'=-\frac{\sum_{w\in V(H)}\phi_{vw}(H,t)^2}{\phi(H,t)^2},
    \]
    that is strictly negative on $t>\lambda_1(H)$. In particular, we have that each term in
    \[\prod_{k=1}^{\ell}\frac{\phi\left((Y-P_{k-1})-v_k,t\right)}{\phi(Y-P_{k-1},t)}\]
    is a strict monotone decreasing function for $t> \lambda_1(Y-v)$.
\end{proof}

\begin{lemma}\label{lemma:broom center ratio with fidget}
    Let $T$ be a tree obtained from a double broom $B(k,d_1,d_2)$ with centers $u,v$ respectively and attaching $\ell$ leaves to one of the leaf neighbour of $v$. Then
    \[
        (d_1-d_2)\lambda(T)^2\ge \ell(d_1-d_2+1) \Longleftrightarrow \frac{\Zv(v)}{\Zv(u)}\le 1.
    \]
    In particular, if $\ell=0$, i.e. we have $T=B(k, d_1, d_2)$, then
    \[
        d_1\geq d_2 \Longleftrightarrow \frac{\Zv(v)}{\Zv(u)}\le 1,
    \]
    with equality if and only if $d_1=d_2$.
\end{lemma}
\begin{proof}
Let $\lambda_1=\lambda_1(T)$. Let $C_u$ be the connected component of $T-v$ containing $u$, and define \[ Y_u=T[V(C_u)\cup\{v\}]. \] Similarly, let $C_v$ be the connected component of $T-u$ containing $v$, and define \[ Y_v=T[V(C_v)\cup\{u\}]. \] Thus both $Y_u$ and $Y_v$ contain $u$ and $v$. Moreover, by the equality case of \cref{lemma:ratio bound}, \[ \frac{\Zv(v)}{\Zv(u)} =\frac{\phi(Y_u-v,\lambda_1)} {\phi_{uv}(Y_u,\lambda_1)}, \qquad \frac{\Zv(u)}{\Zv(v)} =\frac{\phi(Y_v-u,\lambda_1)} {\phi_{uv}(Y_v,\lambda_1)}. \]  Observe that $u$ and $v$ are connected in $Y_u$ and $Y_v$, and let us denote by $P$ the unique path connecting $u$ and $v$ in $Y_u$ and $Y_v$. Let $S$ be a star with exactly $\ell$ leaves. Then
\begin{align*}
    \frac{\Zv(v)}{\Zv(u)}&=\frac{\phi(Y_u-v,\lambda_1)}{\phi(Y_u-P,\lambda_1)}\\
    &=\frac{1}{\lambda_1^{d_1-1}}\left(\lambda_1\phi(Y_u-\{v,u\},\lambda_1) - \sum_{w\in N(u)}\phi(Y_u-\{v,u,w\})\right)\\
    &=\frac{1}{\lambda_1^{d_1-1}}\Bigl( \lambda_1^{d_1}\phi(P-\{u,v\},\lambda_1) -(d_1-1)\lambda_1^{d_1-2}\phi(P-\{u,v\},\lambda_1) \\
    &-\lambda_1^{d_1-1}\phi(P-(N[u]\cup\{v\}),\lambda_1) \Bigr)\\
    &=\frac{1}{\lambda_1}\left((\lambda_1^2-d_1+1)\phi(P-\{u,v\},\lambda_1)-\lambda_1\phi(P-(N[u]\cup\{v\}),\lambda_1) \right)
\end{align*}
In a similar manner we obtain that
\begin{align*}
    \frac{\Zv(u)}{\Zv(v)}&=\frac{\phi(Y_v-u,\lambda_1)}{\phi(Y_v-P,\lambda_1)}\\
    &=\frac{1}{\lambda_1^{d_2-2}\phi(S,\lambda_1)}\left(\lambda_1\phi(Y_v-\{v,u\},\lambda_1) - \sum_{w\in N(v)}\phi(Y_v-\{v,u,w\})\right)\\
    &=\frac{1}{\lambda_1^{d_2-2}\phi(S,\lambda_1)}  \left(\lambda_1^{d_2-1}\phi(S,\lambda_1)\phi(P-\{u,v\},\lambda_1)-\lambda_1^{d_2-2+\ell}\phi(P-\{u,v\},\lambda_1)\right.\\
    &\left.-(d_2-2) \lambda_1^{d_2-3}\phi(S,\lambda_1)\phi(P-\{u,v\},\lambda_1)-\lambda_1^{d_2-1}\phi(S,\lambda_1)\phi(P-(N[v]\cup\{u\}),\lambda_1)\right)\\
    &=\frac{1}{\lambda_1}\left(\left(\lambda_1^2-d_2+2-\frac{\lambda_1^{1+\ell}}{\phi(S,\lambda_1)}\right)\phi(P-\{u,v\},\lambda_1)-\lambda_1\phi(P-(N[v]\cup\{u\}),\lambda_1) \right)
\end{align*}

Since $\frac{\Zv(v)}{\Zv(u)}\ge 0$ we have that $\frac{\Zv(v)}{\Zv(u)}\le 1$ is equivalent to $\frac{\Zv(v)}{\Zv(u)}\le \frac{\Zv(u)}{\Zv(v)}$. In particular,
\begin{align*}
\frac{\Zv(u)}{\Zv(v)}-\frac{\Zv(v)}{\Zv(u)}&=\frac{1}{\lambda_1}\left(\left(d_1-1-d_2+2-\frac{\lambda_1^{1+\ell}}{\phi(S,\lambda_1)}\right)\phi(P-\{u,v\},\lambda_1) \right)\\
&=\frac{1}{\lambda_1}\left(d_1-d_2+1-\frac{\lambda_1^{1+\ell}}{\lambda_1^{1+\ell}-\ell \lambda_1^{\ell-1}}\right)\phi(P-\{u,v\},\lambda_1),
\end{align*}
since we use that $P-(N[u]\cup \{v\})$ is isomorphic to $P-(N[v]\cup \{u\})$.
As $P,S$ are proper subgraphs of $T$, therefore by interlacing we have
\[
\phi(P-\{u,v\},\lambda_1),\phi(S,\lambda_1)>0,
\]

thus
\[
    \frac{\Zv(v)}{\Zv(u)}\le 1 \Longleftrightarrow \frac{\Zv(u)}{\Zv(v)}-\frac{\Zv(v)}{\Zv(u)}\geq 0  \Longleftrightarrow d_1-d_2+1 \ge \frac{\lambda_1^2}{\lambda_1^2-\ell}\Longleftrightarrow (d_1-d_2)\lambda_1^2\ge \ell(d_1-d_2+1).
\]
If $\ell=0$, then this reduces to
\[ \frac{\Zv(v)}{\Zv(u)}\le 1 \Longleftrightarrow \frac{\Zv(u)}{\Zv(v)}-\frac{\Zv(v)}{\Zv(u)}\geq 0  \Longleftrightarrow d_1-d_2\geq 0.\]
This shows that if $\ell=0$ then $\Zv(u)=\Zv(v)$ if and only if we have equality at each place.
\end{proof}

\begin{proof}[Proof of \cref{thm: distance to perron at most k}]
    Suppose, for a contradiction, that there is a principal vertex $v$ with     $d(u,v)>k$, or that $k\geq 1$, $d(u,v)=k$, and $T$ is not a balanced     double broom with centers $u$ and $v$.
    Define a subset $S\subseteq V(T)$ as follows: let $S$ be the set of all vertices of $T$ that do not lie on the shortest path between $u$ and $v$ and are not neighbours of $u$ or $v$.
    Let $Y$ be the subtree of $T$ induced by $V(T)\setminus S$.
    Also let us construct a new tree $T'$ by first removing all edges incident to vertices in $S$ in the original tree $T$, and  connecting each vertex in $S$ directly to $v$ by a new edge, see \cref{fig: example construction new tree} for a graphical example of this.
    Thus, each vertex in $S$ becomes a  leaf attached to $v$, thus the resulting graph $T'$ is a double broom tree with centers $u$ and $v$. Observe that $Y$ is still a spanning subgraph of $T'$.
    Observe that $T'$ can be obtained from $T$ by an application of \cref{lem:general-ERL}, taking the principal vertex $v$ of $T$ as the vertex $u$ in  \cref{lem:general-ERL} and letting $E'$ be the set of edges incident to $S$ (these playing the role of the edges $v_iw_i$ in \cref{lem:general-ERL} with $d(v_i,u)<d(w_i,u)$).
    From this we conclude that $\lambda_1(T')>\lambda_1(T)$ if $T$ is not isomorphic to $T'$. Since $T'$ is a double broom tree on $n$ vertices, we have the following identity
    \[
    n=(\deg_{T'}(u)-1)+(\deg_{T'}(v)-1)+(d+1),
    \]
    where $(\deg_{T'}(u)-1)+(\deg_{T'}(v)-1)$ counts the leaves in $T'$ and $d+1$ counts the vertices on the path from $u$ to $v$. Since $\deg_{T'}(u)=\deg_{T}(u)$, this can be rearranged to
    \[
    \deg_{T'}(v)=n - \deg_{T}(u)- d+1.
    \]
    Additionally, we have $\deg_{T'}(v)\leq \deg_{T'}(u)$, since
    \[
    \deg_{T'}(v)= n - \deg_{T}(u) - d+ 1\leq n-\frac{n-k+1}{2}-k+1=\frac{n-k+1}{2}\le \deg_{T'}(u).
    \]
    By assumption $v$ is a principal vertex of $T$, thus $1\le \frac{\Zv(v)}{\Zv(u)}$ and we have the following chain of inequalities
    \[
    1\le\frac{\Zv(v)}{\Zv(u)}\le \frac{\phi_v(Y,\lambda_1(T))}{\phi_{uv}(Y,\lambda_1(T))} \le \frac{\phi_v(Y,\lambda_1(T'))}{\phi_{uv}(Y,\lambda_1(T'))}=\frac{\Zv'(v)}{\Zv'(u)}\le 1.
    \]
    The second inequality follows from Lemma~\ref{lemma:ratio bound}, the third inequality follows from Lemma~\ref{lemma: monotonicity of ratio} and the last inequality follows from Lemma~\ref{lemma:broom center ratio with fidget} with $\ell=0$, since $\deg_{T'}(u)\geq \deg_{T'}(v)$ by assumption.
    Therefore equality holds throughout the chain. In particular, the last inequality is an equality, and \cref{lemma:broom center ratio with fidget} implies \[ \deg_{T'}(u)=\deg_{T'}(v). \] Together with \[ \deg_{T'}(v) =n-\deg_T(u)-d+1 \leq \frac{n-k+1}{2} \leq \deg_{T'}(u), \] this forces equality at every step, and hence $d=k$ and \[ \deg_T(u)=\deg_{T'}(v)=\frac{n-k+1}{2}. \] Equality in the monotonicity step also forces $\lambda_1(T')=\lambda_1(T)$. Since a nontrivial application of \cref{lem:general-ERL} strictly increases the spectral radius, the construction of $T'$ must have been trivial. Thus $T=T'$, and $T$ is a balanced double broom with centers $u$ and $v$ at distance $k$.
    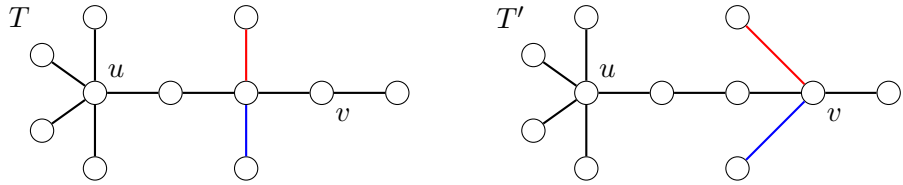
\begin{figure}[htb]
    \centering
    \begin{tikzpicture}[main_node/.style={circle,draw,minimum size=0.8em,inner sep=3pt}]

        \node[] at (-1,1) {$T$};
        \node[main_node] (0) at (0,1) {};
        \node[main_node] (1) at (0,-1) {};
        \node[main_node] (2) at (-0.7,-0.5) {};
        \node[main_node] (3) at (-0.7, 0.5) {};
        \node[main_node] (4) at (0,0) {};
        \node[above, yshift=2pt, xshift=8pt] at (4) {$u$};
        \node[main_node] (5) at (1,0) {};
        \node[main_node] (6) at (2,0) {};
        \node[main_node] (7) at (3,0) {};
        \node[below, yshift=-2pt, xshift=8pt] at (7) {$v$};
        \node[main_node] (8) at (2,1) {};
        \node[main_node] (9) at (2,-1) {};
        \node[main_node] (10) at (4,0) {};
        \path[draw, thick]
        (4) edge node {} (3)
        (4) edge node {} (2)
        (4) edge node {} (1)
        (4) edge node {} (0)
        (4) edge node {} (5)
        (5) edge node {} (6)
        (6) edge[red] node {} (8)
        (6) edge[blue] node {} (9)
        (6) edge node {} (7)
        (7) edge node {} (10);

        \begin{scope}[xshift=6.5cm]
            \node[] at (-1,1) {$T'$};
            \node[main_node] (11) at (-0.7,0.5) {};
            \node[main_node] (12) at (0,0) {};
            \node[above, yshift=2pt, xshift=8pt] at (12) {$u$};
            \node[main_node] (13) at (-0.7,-0.5) {};
            \node[main_node] (14) at (0,1) {};

            \node[main_node] (15) at (0,-1) {};
            \node[main_node] (16) at (1,0) {};
            \node[main_node] (17) at (2,0) {};
            \node[main_node] (18) at (2,1) {};
            \node[main_node] (19) at (3,0) {};
            \node[below, yshift=-2pt, xshift=8pt] at (19) {$v$};
            \node[main_node] (20) at (4,0) {};
            \node[main_node] (21) at (2,-1) {};

            \path[draw, thick]
            (12) edge node {} (14)
            (12) edge node {} (13)
            (12) edge node {} (11)
            (12) edge node {} (15)
            (12) edge node {} (16)
            (16) edge node {} (17)
            (18) edge[red] node {} (19)
            (21) edge[blue] node {} (19)
            (19) edge node {} (20)
            (17) edge node {} (19);
        \end{scope}

    \end{tikzpicture}
    \caption{On the left $T$ and on the right $T'$ where we coloured the edges incident to the subset $S$.}
    \label{fig: example construction new tree}
\end{figure}
\end{proof}

\begin{theorem}\label{thm:large degree is Perron}
    Let $T$ be a tree on $n$ vertices with maximum degree $\Delta(T)\ge n/2$. Let $\Zv$ be the Perron vector of $T$. If $u$ is a vertex of maximum degree, then $u$ is a principal vertex, i.e., $\Zv(u)$ is a maximum entry of $\Zv$.
\end{theorem}
\begin{proof}
    Observe that $T$ has a unique maximum degree vertex if and only if $\Delta(T)>\frac{n}{2}$ or $\Delta(T)=\frac{n}{2}$ and $T\ncong B(1,n/2,n/2)$.
    Therefore, if $T\ncong B(1,n/2,n/2)$, then by \cref{thm: distance to perron at most k} the unique maximum degree vertex $v$ of degree at least $n/2$ must be the unique principal vertex of $T$.

    If $T\cong B(1,n/2,n/2)$, then we have exactly two non-leaf vertex of the same degree $n/2$. By symmetry, the two non-leaf vertices are both principal vertices.
\end{proof}

\section{Sustainability of trees with large maximum degree}
In this section we study sustainable trees whose maximum degree is at least \(n/2\). We first prove that if
\[
    \Delta(T)\geq \frac n2+1,
\]
then \(T\) is non-sustainable. The boundary case \(\Delta(T)=n/2\) is more delicate; after replacing a leaf edge incident to a maximum-degree vertex, the principal vertex of the resulting tree need not remain the original maximum-degree vertex. We therefore treat the cases \(\Delta(T)\geq n/2+1\) and \(\Delta(T)=n/2\) separately.
The main idea of the proof is that by reconnecting an edge next to a vertex of maximum degree, it doesn't loose it property of being a principal vertex. Unfortunately, the  case when $\Delta(T)=n/2$ this argument won't be applicable for all possible trees. To this end we define an exceptional family of trees described by Figure~\ref{fig: exceptional family} $\mathcal{F}=\{F_1(n),\dots,F_8(n)\}$ where the stability of being a principal vertex could fail and analyze their sustainability separately.

\subsection{Trees with maximum degree at least \texorpdfstring{$n/2+1$}{n/2+1}}
\begin{theorem}\label{thm: large degree is non-sustainable}
    Let $T$ be a tree on $n$ vertices with maximum degree $\Delta(T)\geq \frac{1}{2}n+1$, then $T$ is non-sustainable.
\end{theorem}
\begin{proof}
    Let $u$ be a vertex of degree $\Delta=\Delta(T)$. First, we show that $u$ has a leaf neighbour. Indeed, if every neighbour of $u$ has degree at least two, then
    \[
       2|E(T)|=  2(n-1)\ge \Delta+\Delta\cdot 2 + (n-\Delta-1)\cdot 1\ge 2n+1
    \]
    a contradiction.

    Second, we claim that any  vertex $w\neq u$ has degree at most $n/2-1$; if $w\neq u$ has degree at least $n/2$, then the sum of the degrees would be
    \[
        2(n-1)\ge \Delta+n/2+(n-2)\ge 2n-1
    \]
    a contradiction.

    Now let $v$ be a leaf neighbour of $u$ and set $e=uv$.  Choose any vertex $w\neq u,v$, and form $T'=T-uv+vw$.
    Then $T'$ is a tree and has maximum degree $\Delta-1\geq\frac{1}{2}n$, and so $u$ is a vertex of maximum degree in $T'$. By Theorem~\ref{thm:large degree is Perron} we have that $u$ is a principal vertex. Now $T$ can be obtained from $T'$ by replacing the edge $e'=vw$ with $e=uv$, thus \cref{lem: ERL} gives us $\lambda_1(T)>\lambda_1(T')$. This implies that our initial edge $e$ is non-sustainable, and so $T$ is non-sustainable.
\end{proof}

\begin{remark}
The bound in \cref{thm: large degree is non-sustainable} is sharp. For $k\geq 3$, let
\[
    T_k=B(2,k,k-1).
\]
Then $T_k$ has $n=2k$ vertices and maximum degree $k=n/2$, but $T_k$ is sustainable.

Let $u$ be the center of degree $k$, and let $v$ be the center of degree $k-1$. By \cref{thm:large degree is Perron}, $u$ is a principal vertex. Hence every edge not incident to $u$ is sustainable by \cref{lem:ERLsustainable}. It remains to check the edges incident to $u$.

If $x$ is a leaf neighbour of $u$, then
\[
    T_k-ux+xv\cong T_k,
\]
so the edge $ux$ is sustainable. The remaining edge incident to $u$ lies on the path from $u$ to $v$ and is incident to a vertex of degree $2$. Hence it is sustainable by \cref{lem: incident degree 2 is sustainable}. Therefore $T_k$ is sustainable.
\end{remark}

\subsection{Trees with maximum degree \texorpdfstring{$n/2$}{n/2}}

\begin{lemma}\label{lemma:reduction}
    Let $T_0$ be a tree, let $v\in V(T_0)$, and let $\ell_1,\dots,\ell_k$ be leaf neighbours of $v$. For $m\geq 1$, let $\mathcal F_m$ be the family of trees obtained from $T_0$ as follows: for each $i$, identify $\ell_i$ with a vertex of a tree $T_i$, and assume that
    \[
        \sum_{i=1}^k |E(T_i)|=m.
    \]
    Then the maximum of $\lambda_1(F)$ over $F\in\mathcal F_m$ is attained by a tree obtained from $T_0$ by attaching all $m$ new edges as a star at one of the vertices $\ell_i$.

\end{lemma}
\begin{proof}
    Let $F\in\mathcal F_m$ maximize $\lambda_1(F)$. Suppose that, for some $i$, the attached tree $T_i$ has an edge not incident to $\ell_i$. Applying a generalized tree shift toward $\ell_i$ strictly increases the spectral radius and keeps the resulting tree in $\mathcal F_m$. This contradicts the maximality of $F$. Hence each non-empty $T_i$ must be a star centered at $\ell_i$.

    It remains to show that the new star edges may be concentrated at a single one of the vertices $\ell_i$. Let $\Zv$ be the Perron vector of $F$, and choose $\ell_i$ such that
    \[
        \Zv(\ell_i)=\max\{\Zv(\ell_1),\dots,\Zv(\ell_k)\}.
    \]
    For every $j\neq i$, rotate all new leaf edges attached at $\ell_j$ to $\ell_i$. By \cref{lem:general-ERL}, this does not decrease the spectral radius. Therefore there exists a maximizer in which all $m$ new edges are attached as a star centered at a single vertex $\ell_i$.
\end{proof}

\begin{figure}[htb]
    \centering
    \begin{tikzpicture}[main_node/.style={circle,draw,minimum size=0.8em,inner sep=3pt}]

        \begin{scope}[xshift=-2cm]
        \node[] at (-0.7, 1.5) {$F_1(n)$};
        \node[main_node] (0) at (0,1) {};
        \node[main_node] (1) at (0,-1) {};
        \node[main_node] (2) at (-0.7,-0.5) {};
        \node[main_node] (3) at (-0.7, 0.5) {};
        \node[main_node, fill=black] (4) at (0,0) {};
        \node[above, yshift=2pt, xshift=7pt] at (4) {$u$};
        \node[main_node] (5) at (-1,0) {};
        \node[main_node] (6) at (1,0) {};
        \node[draw, rectangle, fill=black] (7) at (2,0) {};
        \node[above, yshift=2pt, xshift=-7pt] at (7) {$v$};
        \node[main_node] (8) at (2,1) {};
        \node[main_node] (9) at (2,-1) {};
        \node[main_node] (10) at (2.7, -0.5) {};
        \node[main_node] (11) at (2.7 ,0.5) {};

        \path[draw, thick]
        (4) edge node {} (3)
        (4) edge node {} (2)
        (4) edge node {} (1)
        (4) edge node {} (0)
        (4) edge node {} (5)
        (4) edge node {} (6)
        (7) edge node {} (8)
        (7) edge node {} (9)
        (6) edge node {} (7)
        (7) edge node {} (10)
        (7) edge node {} (11);
        \end{scope}

        \begin{scope}[xshift=3cm]
        \node[] at (-0.7, 1.5) {$F_2(n)$};
        \node[main_node] (0) at (0,1) {};
        \node[main_node] (1) at (0,-1) {};
        \node[main_node] (2) at (-0.7,-0.5) {};
        \node[main_node] (3) at (-0.7, 0.5) {};
        \node[main_node, fill=black] (4) at (0,0) {};
        \node[above, yshift=2pt, xshift=7pt] at (4) {$u$};
        \node[main_node] (5) at (-1,0) {};
        \node[main_node] (6) at (1,0) {};
        \node[regular polygon,
        regular polygon sides=3,
        draw,
        inner sep=2pt,
        fill=black] (7) at (2,0) {};
        \node[above, yshift=2pt, xshift=-7pt] at (7) {$v$};
        \node[main_node] (8) at (2,1) {};
        \node[main_node] (9) at (2,-1) {};
        \node[main_node] (10) at (3, 0) {};
        \node[main_node] (11) at (2.7 ,1) {};

        \path[draw, thick]
        (4) edge node {} (3)
        (4) edge node {} (2)
        (4) edge node {} (1)
        (4) edge node {} (0)
        (4) edge node {} (5)
        (4) edge node {} (6)
        (7) edge node {} (8)
        (7) edge node {} (9)
        (6) edge node {} (7)
        (7) edge node {} (10)
        (8) edge node {} (11);
        \end{scope}

        \begin{scope}[xshift=8cm]
        \node[] at (-0.7, 1.5) {$F_3(n)$};
        \node[main_node] (0) at (0,1) {};
        \node[main_node] (1) at (0,-1) {};
        \node[main_node] (2) at (-0.7,-0.5) {};
        \node[main_node] (3) at (-0.7, 0.5) {};
        \node[main_node, fill=black] (4) at (0,0) {};
        \node[above, yshift=2pt, xshift=7pt] at (4) {$u$};
        \node[main_node] (5) at (-1,0) {};
        \node[main_node, fill=black] (6) at (1,0) {};
        \node[above, yshift=2pt, xshift=-7pt] at (6) {$v$};
        \node[main_node] (7) at (2,0) {};
        \node[main_node] (8) at (1,1) {};
        \node[main_node] (9) at (1,-1) {};
        \node[main_node] (10) at (1.7, 0.5) {};
        \node[main_node] (11) at (1.7, -0.5) {};

        \path[draw, thick]
        (4) edge node {} (3)
        (4) edge node {} (2)
        (4) edge node {} (1)
        (4) edge node {} (0)
        (4) edge node {} (5)
        (4) edge node {} (6)
        (6) edge node {} (8)
        (6) edge node {} (9)
        (6) edge node {} (7)
        (6) edge node {} (11)
        (6) edge node {} (10);

        \end{scope}

        \begin{scope}[xshift=-2cm, yshift=-3.5cm]
        \node[] at (-0.7, 1.5) {$F_4(n)$};
        \node[main_node] (0) at (0,1) {};
        \node[main_node] (1) at (0,-1) {};
        \node[main_node] (2) at (-0.7,-0.5) {};
        \node[main_node] (3) at (-0.7, 0.5) {};
        \node[main_node, fill=black] (4) at (0,0) {};
        \node[above, yshift=2pt, xshift=7pt] at (4) {$u$};
        \node[main_node] (5) at (-1,0) {};
        \node[draw, rectangle, fill=black] (6) at (1,0) {};
        \node[above, yshift=2pt, xshift=-7pt] at (6) {$v$};
        \node[main_node] (7) at (1.7,0.5) {};
        \node[main_node] (8) at (1,1) {};
        \node[main_node] (9) at (1,-1) {};
        \node[main_node] (10) at (2.7,0.5) {};
        \node[main_node] (11) at (1.7,-0.5) {};

        \path[draw, thick]
        (4) edge node {} (3)
        (4) edge node {} (2)
        (4) edge node {} (1)
        (4) edge node {} (0)
        (4) edge node {} (5)
        (4) edge node {} (6)
        (6) edge node {} (8)
        (6) edge node {} (9)
        (6) edge node {} (7)
        (7) edge node {} (10)
        (6) edge node {} (11);

        \end{scope}

        \begin{scope}[xshift=4cm, yshift=-3.5cm]
        \node[] at (-1.7, 1.5) {$F_5(n)$};
        \node[main_node] (0) at (0,1) {};
        \node[main_node] (1) at (0,-1) {};
        \node[main_node] (2) at (-0.7,-0.5) {};
        \node[main_node] (3) at (-0.7, 0.5) {};
        \node[main_node, fill=black] (4) at (0,0) {};
        \node[above, yshift=2pt, xshift=7pt] at (4) {$u$};
        \node[main_node] (5) at (-1,0) {};
        \node[draw, rectangle, fill=black] (6) at (1,0) {};
        \node[above, yshift=2pt, xshift=-7pt] at (6) {$v$};
        \node[main_node] (7) at (1.7,-0.5) {};
        \node[main_node] (8) at (1,1) {};
        \node[main_node] (9) at (1,-1) {};
        \node[main_node] (10) at (1.7,0.5) {};
        \node[main_node] (11) at (-2,0) {};

        \path[draw, thick]
        (4) edge node {} (3)
        (4) edge node {} (2)
        (4) edge node {} (1)
        (4) edge node {} (0)
        (4) edge node {} (5)
        (4) edge node {} (6)
        (6) edge node {} (8)
        (6) edge node {} (9)
        (6) edge node {} (7)
        (5) edge node {} (11)
        (6) edge node {} (10);
        \end{scope}

        \begin{scope}[xshift=8cm, yshift=-3.5cm]
        \node[] at (-0.7, 1.5) {$F_6(n)$};
        \node[main_node] (0) at (0,1) {};
        \node[main_node] (1) at (0,-1) {};
        \node[main_node] (2) at (-0.7,-0.5) {};
        \node[main_node] (3) at (-0.7, 0.5) {};
        \node[main_node, fill=black] (4) at (0,0) {};
        \node[above, yshift=2pt, xshift=7pt] at (4) {$u$};
        \node[main_node] (5) at (-1,0) {};
        \node[regular polygon,
        regular polygon sides=3,
        draw,
        inner sep=2pt,
        fill=black] (6) at (1,0) {};
        \node[above, yshift=2pt, xshift=-7pt] at (6) {$v$};
        \node[main_node] (7) at (2,0) {};
        \node[main_node] (8) at (1,1) {};
        \node[main_node] (9) at (1,-1) {};
        \node[main_node] (10) at (3,0) {};
        \node[main_node] (11) at (2,1) {};

        \path[draw, thick]
        (4) edge node {} (3)
        (4) edge node {} (2)
        (4) edge node {} (1)
        (4) edge node {} (0)
        (4) edge node {} (5)
        (4) edge node {} (6)
        (6) edge node {} (8)
        (6) edge node {} (9)
        (6) edge node {} (7)
        (8) edge node {} (11)
        (7) edge node {} (10);

        \end{scope}

        \begin{scope}[xshift=-2cm, yshift=-7cm]
        \node[] at (-0.7, 1.5) {$F_7(n)$};
        \node[main_node] (0) at (0,1) {};
        \node[main_node] (1) at (0,-1) {};
        \node[main_node] (2) at (-0.7,-0.5) {};
        \node[main_node] (3) at (-0.7, 0.5) {};
        \node[main_node, fill=black] (4) at (0,0) {};
        \node[above, yshift=2pt, xshift=7pt] at (4) {$u$};
        \node[main_node] (5) at (-1,0) {};
        \node[regular polygon,
        regular polygon sides=3,
        draw,
        inner sep=2pt,
        fill=black] (6) at (1,0) {};
        \node[above, yshift=2pt, xshift=-7pt] at (6) {$v$};
        \node[main_node] (7) at (2,0) {};
        \node[main_node] (8) at (1,1) {};
        \node[main_node] (9) at (1,-1) {};
        \node[main_node] (10) at (2.7,0.5) {};
        \node[main_node] (11) at (2.7,-0.5) {};

        \path[draw, thick]
        (4) edge node {} (3)
        (4) edge node {} (2)
        (4) edge node {} (1)
        (4) edge node {} (0)
        (4) edge node {} (5)
        (4) edge node {} (6)
        (6) edge node {} (8)
        (6) edge node {} (9)
        (6) edge node {} (7)
        (7) edge node {} (11)
        (7) edge node {} (10);

        \end{scope}

        \begin{scope}[xshift=2.7cm, yshift=-7cm]
        \node[] at (-0.7, 1.5) {$F_8(n)$};
        \node[main_node] (0) at (0,1) {};
        \node[main_node] (1) at (0,-1) {};
        \node[main_node] (2) at (-0.7,-0.5) {};
        \node[main_node] (3) at (-0.7, 0.5) {};
        \node[main_node, fill=black] (4) at (0,0) {};
        \node[above, yshift=2pt, xshift=7pt] at (4) {$u$};
        \node[main_node] (5) at (-1,0) {};
        \node[regular polygon,
        regular polygon sides=3,
        draw,
        inner sep=2pt,
        fill=black] (6) at (1,0) {};
        \node[above, yshift=2pt, xshift=-7pt] at (6) {$v$};
        \node[main_node] (7) at (2,0) {};
        \node[main_node] (8) at (1,1) {};
        \node[main_node] (9) at (1,-1) {};
        \node[main_node] (10) at (3,0) {};
        \node[main_node] (11) at (4,0) {};

        \path[draw, thick]
        (4) edge node {} (3)
        (4) edge node {} (2)
        (4) edge node {} (1)
        (4) edge node {} (0)
        (4) edge node {} (5)
        (4) edge node {} (6)
        (6) edge node {} (8)
        (6) edge node {} (9)
        (6) edge node {} (7)
        (10) edge node {} (11)
        (7) edge node {} (10);

        \end{scope}
        \begin{scope}[xshift=7.4cm, yshift=-6.5cm]
            \node[main_node, fill=black] at (0,0) {};
            \node[right] at (0.3,0) {vertex of degree $\frac{n}{2}$};

            \node[draw, rectangle, fill=black] at (0,-0.6) {};
            \node[right] at (0.3,-0.6) {vertex of degree $\frac{n}{2} - 1$};

            \node[
            regular polygon,
            regular polygon sides=3,
            draw,
            inner sep=2pt,
            fill=black] at (0,-1.2) {};
            \node[right] at (0.3,-1.2) {vertex of degree $\frac{n}{2} - 2$};

        \end{scope}
    \end{tikzpicture}
    \caption{Here we define the family $\mathcal{F}$.  The degrees of the marked vertices are dependent on $n$ according to the legend, while all the other vertices have degree $1, 2$ or $3$ as depicted in the figure.}
    \label{fig: exceptional family}
\end{figure}
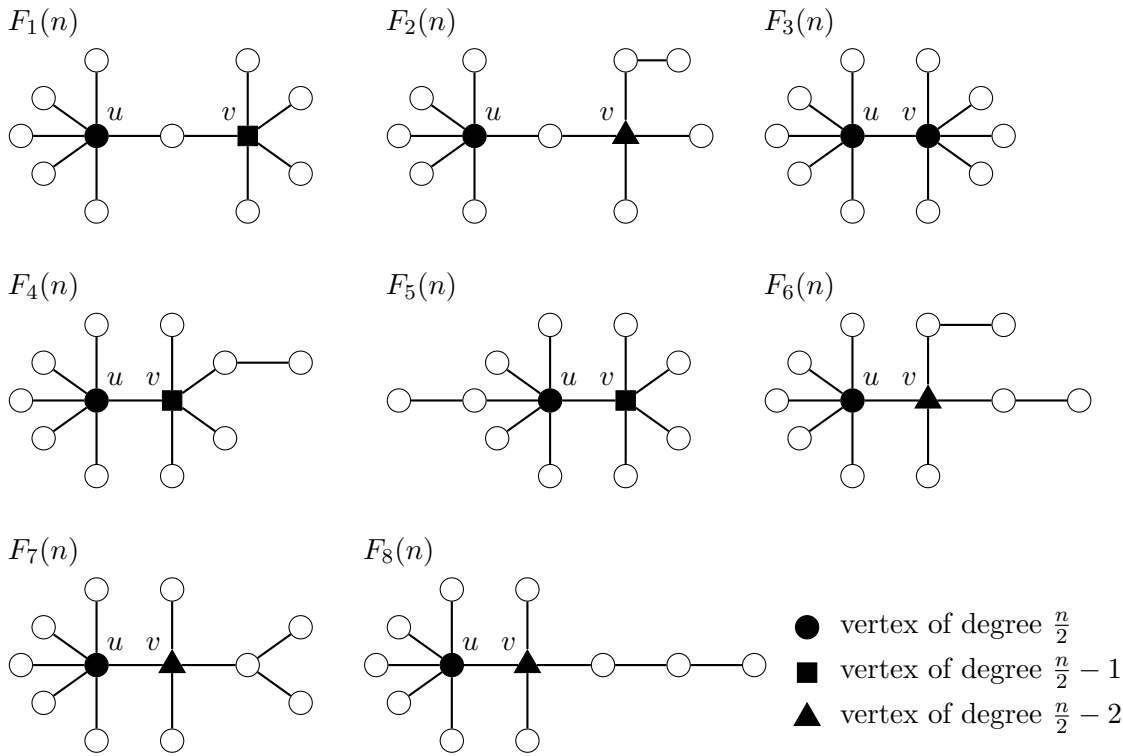

\begin{proposition}\label{thm: degree n/2 non-sustainable}
Let $T$ be a tree on $n\geq 12$ vertices with maximum degree
\[
    \Delta(T)=\tfrac n2.
\]
If $T\notin\mathcal F$, where $\mathcal F$ is the family defined in \cref{fig: exceptional family}, then $T$ is non-sustainable.
\end{proposition}
\begin{proof}
    For the sake of contradiction let us assume that there exists a $T$, such that $T\notin\mathcal{F}$ and $T$ is sustainable. Since $\Delta(T)\ge n/2$, it means that at the largest degree vertex $u$, we have a leaf edge. This leaf edge can be replaced to $v$, such that $\lambda_1(T')\ge \lambda_1(T)$. Without loss of generality, we may assume that $\Zv_{T'}(u)<\Zv_{T'}(v)$ and that $\Zv_{T'}(v)$ is maximal in $T'$. Indeed, since otherwise using iteratively \cref{lem: ERL} we get such a $v$ and $T'$. By \cref{thm: distance to perron at most k} we have that $1\le d_{T'}(u,v)\le 3$.

    {\bf Case $d_{T'}(u,v)=3$.} By \cref{thm: distance to perron at most k} we know that $T'$ has to be a balanced double broom with centers $u,v$, i.e. $T'=B(3,n/2-1,n/2-1)$. But then \cref{lemma:broom center ratio with fidget} implies that $\Zv_{T'}(u)=\Zv_{T'}(v)$ contradicting the choice of $v$.

    {\bf Case $d_{T'}(u,v)=2$.} Let $w$ be the middle vertex of the path connecting $u$ and $v$.  Now we will investigate 3 sub-cases. In each case we will argue similarly, first we identify a possible subtree $Y$ containing $u,v$ in $T'$, then we apply \cref{lem:general-ERL} to reconnect every vertex outside of $Y$ directly to $v$. This way we will obtain a new tree $T''$ that has $Y$ as a subgraph, moreover $\lambda_1(T'')\ge \lambda_1(T')$ (with equality if and only if $T'= T''$). Then we combine these with \cref{lemma:ratio bound} and \cref{lemma: monotonicity of ratio} to obtain the contradiction. Since $T\notin\mathcal{F}$ we have the following 3 local neighbourhood around $u, v$.
    \begin{itemize}
        \item If $w$ has a neighbour $z$ in $T'$ other than $u,v$. Then let $Y$ be the tree spanned $N_{T'}[u]\cup \{v,w,z\}$ and $T''$ the tree obtained from $T'$ by applying \cref{lem:general-ERL} on all the edges not contained in $Y$. Then we have
        \[
        1< \frac{\Zv_{T'}(v)}{\Zv_{T'}(u)}\le \frac{\phi_v(Y,\lambda_1(T'))}{\phi_{uv}(Y,\lambda_1(T'))}\le \frac{\phi_v(Y,\lambda_1(T''))}{\phi_{uv}(Y,\lambda_1(T''))}.
        \]
        By the second part of \cref{lemma:ratio bound}, we know that $\frac{\phi_v(Y,\lambda_1(T''))}{\phi_{uv}(Y,\lambda_1(T''))}=\frac{\Zv_{T''}(v)}{\Zv_{T''}(u)}$, that is equal to $1$ by the symmetry of $T''$, thus it is a contradiction.
        \item  There is a vertex $z$ at distance $2$ from $u$ that is not a neighbour  of $w$; say $z$ is adjacent to $y\in N_{T'}(u)$, $y\neq w$. Let $Y$ be the subtree  spanned by $N_{T'}[u]\cup\{z,w,v\}$, and let $T''$ be obtained from $T'$ by  applying \cref{lem:general-ERL} to reconnect every vertex outside $Y$ to $v$.  Then $T''$ is the balanced double broom $B(2,n/2-1,n/2-1)$ with centers $u$  and $v$, together with one extra leaf, namely $z$, attached to the leaf  neighbour $y$ of $u$. As before,
  \[
      1<\frac{\Zv_{T'}(v)}{\Zv_{T'}(u)}
      \le\frac{\phi_v(Y,\lambda_1(T'))}{\phi_{uv}(Y,\lambda_1(T'))}
      \le\frac{\phi_v(Y,\lambda_1(T''))}{\phi_{uv}(Y,\lambda_1(T''))}
      =\frac{\Zv_{T''}(v)}{\Zv_{T''}(u)}.
  \]
  On the other hand, apply \cref{lemma:broom center ratio with fidget} to $T''$  with the roles of $u$ and $v$ exchanged, so that the fidget lies on the side  of $u$. Since $d_1=d_2=n/2-1$ and $\ell=1$, the condition  $(d_1-d_2)\lambda_1(T'')^2\ge\ell(d_1-d_2+1)$ reads $0\ge1$, which is false;  hence $\Zv_{T''}(u)/\Zv_{T''}(v)>1$, that is,  $\Zv_{T''}(v)/\Zv_{T''}(u)<1$, contradicting the  above.

        \item  All the vertices in $N(u)\setminus \{w\}$ are leaves of $T'$. Let us denote by $Y_1$ the tree spanned by $N_{T'}[u]\cup\{v\}$. Since $T\notin \mathcal{F}$, therefore we know that there is a subtree $Y_2$ of $T'-w$ containing $v$ that has exactly 2 edges that is not incident to $v$. Let $Y=Y_1\cup Y_2$ and apply \cref{lem:general-ERL} to reconnect all edges outside of $Y$ to $v$ to obtain $T''$, thus $\lambda_1(T'')\ge \lambda_1(T)$ and the vertex $v$ has degree $n/2-2$. By \cref{lemma:reduction} we have that $\lambda_1(T'')\le \lambda_1(T''_c)$, where $T''_c$ is a tree obtained from $B(2,n/2-1,n/2-2)$ by attaching 2 leaves to a leaf of $v$. Similarly as before we have
        \[
            1<\frac{\Zv_{T'}(v)}{\Zv_{T'}(u)}= \frac{\phi_v(Y,\lambda_1(T'))}{\phi_{uv}(Y,\lambda_1(T'))} \le \frac{\phi_v( Y,\lambda_1(T''_c))}{\phi_{uv}( Y,\lambda_1(T''_c))}= \frac{\Zv_{T''_c}(v)}{\Zv_{T''_c}(u)}.
        \]
        By \cref{lemma:broom center ratio with fidget} we know that $\Zv_{T''_c}(v)/\Zv_{T''_c}(u)\le 1$ if and only if $1\cdot \lambda_1(T''_c)^2\ge 2\cdot 2$, which is true, since $\lambda_1(T''_c)\ge \lambda_1(T)\ge \sqrt{n/2}\ge 2$. Thus again a contradiction.
    \end{itemize}

    \textbf{Case $d_{T'}(u,v)=1$.} In this case we will argue similarly to before.
    \begin{itemize}
    \item By deleting the edge $uv$ in $T'$, the component of $u$  has at least $n/2-2+2$ edges.
    Let $Y$  be the induced subtree of $T'-uv$ with vertices $N_{T'}[u]\cup\{z_1,z_2\}$, where $z_1, z_2\notin N_{T'}[u]$.  Now apply \cref{lem:general-ERL} to connect everyone in the complement of $Y$ directly to $v$. This tree is $T''$ and we know that
    \[
    1<\frac{\phi_v(Y,\lambda_1(T''))}{\phi_{uv}(Y,\lambda_1(T''))}=\frac{\Zv_{T''}(v)}{\Zv_{T''}(u)}.
    \]
    Now let $Y''$ be the subtree induced of $T''$ by the vertices $N_{T''}[v]\cup N_{T''}[u]$. The tree $Y''$ is a balanced double star on $n-2$ vertices and $\lambda_1(T'')\ge\lambda_1(Y'')$. Thus by \cref{lemma:ratio bound} and \cref{lemma: monotonicity of ratio} we have
    \begin{align*}
        \frac{\Zv_{T''}(v)}{\Zv_{T''}(u)}&=\frac{\phi_{uv}(Y'',\lambda_1(T''))}{\phi_{u}(Y'',\lambda_1(T''))}\le \frac{\phi_{uv}(Y'',\lambda_1(Y''))}{\phi_{u}(Y'',\lambda_1(Y''))}=\frac{\Zv_{Y''}(v)}{\Zv_{Y''}(u)}=1.
    \end{align*}
    Thus it's a contradiction.

    \item By deleting the edge $uv$ in $T'$, the component of $u$ is has exactly $n/2-2+1$ edges. Since $T\notin\mathcal{F}$ this means that $v$ and $u$ have a neighbour in $T'-uv$ of degree at least $2$, say $z_1$ (resp. $z_2$). Let $s_1\neq v$ be a neighbour of $z_1$ and $s_2\neq u$ be a neighbour of $z_2$. Let $Y$  be the subtree of $T'$ spanned by $N_{T'}[u]\cup\{z_1,s_1,s_2\}$  and let $T''$ the tree obtained from $T'$ by applying \cref{lem:general-ERL} to reconnect every edge outside of $Y$ to $v$.  Then
    \[
        1<\frac{\Zv_{T'}(v)}{\Zv_{T'}(u)}= \frac{\phi_v(Y,\lambda_1(T'))}{\phi_{uv}(Y,\lambda_1(T'))} \le \frac{\phi_v( Y,\lambda_1(T''))}{\phi_{uv}( Y,\lambda_1(T''))}= \frac{\Zv_{T''}(v)}{\Zv_{T''}(u)}=1,
    \]
    where in the last inequality we used the fact that $T''$ is symmetric. Thus it is a contradiction.
    \item All the vertices in $N(u)\setminus \{v\}$ are leaves of $T'$. Let us denote by $Y_1$ be the tree induced by $N_{T'}[u]\cup\{v\}$ in $T'$.
    Since $T\notin \mathcal{F}$, therefore we know that there is a subtree $Y_2$ of $T'-u$ containing $v$ that has exactly 3 edges  that are not incident to $v$.
    Let $Y=Y_1\cup Y_2$ and apply \cref{lem:general-ERL} to reconnect all edges outside of $Y$ to $v$ to obtain $T''$, thus $\lambda_1(T'')\ge \lambda_1(T)$ and the vertex $v$ has degree $n/2-2$. By \cref{lemma:reduction} we have that $\lambda_1(T'')\le \lambda_1(T''_c)$, where $T''_c$ is a tree obtained from $B(1,n/2-1,n/2-2)$ by attaching 3 leaves to a leaf of $v$. Similarly as before we have
        \[
            1<\frac{\Zv_{T'}(v)}{\Zv_{T'}(u)}= \frac{\phi_v(Y,\lambda_1(T'))}{\phi_{uv}(Y,\lambda_1(T'))} \le \frac{\phi_v( Y,\lambda_1(T''_c))}{\phi_{uv}( Y,\lambda_1(T''_c))}= \frac{\Zv_{T_c''}(v)}{\Zv_{T_c''}(u)}.
        \]
        By \cref{lemma:broom center ratio with fidget} we know that $\Zv_{T_c''}(v)/\Zv_{T_c''}(u)\le 1$ if and only if $1\cdot \lambda_1(T''_c)^2\ge 3\cdot 2$, which is true, since $\lambda_1(T''_c)\ge \lambda_1(T)\ge \sqrt{n/2}\ge \sqrt{6}$. Thus, it's a contradiction.
    \end{itemize}
\end{proof}

In the next statement we characterize all sustainable trees with maximum degree exactly $n/2$. For the proof we need the following lemmas.
\begin{lemma}\label{lemma:cheap_trick}
    Let $T$ be a tree, let $u\in V(T)$, and suppose that $u$ is a principal vertex. If $u$ has $\ell$ leaf neighbours, then
    \[
        \lambda_1(T)+\ell\left(1-\frac{1}{\lambda_1(T)}\right)\leq \deg(u).
    \]
\end{lemma}
\begin{proof}
    Let $\Zv$ be the Perron vector of $T$, and write $\lambda_1=\lambda_1(T)$. If $v$ is a leaf neighbour of $u$, then
    \[
        \lambda_1\Zv(v)=\Zv(u),
    \]
    so
    \[
        \Zv(v)=\frac{1}{\lambda_1}\Zv(u).
    \]
    Since $u$ is a principal vertex, every non-leaf neighbour $w$ of $u$ satisfies
    \[
        \Zv(w)\leq \Zv(u).
    \]
    Therefore
    \[
        \lambda_1\Zv(u)
        =
        \sum_{w\sim u}\Zv(w)
        \leq
        \ell\frac{1}{\lambda_1}\Zv(u)+(\deg(u)-\ell)\Zv(u).
    \]
    Dividing by $\Zv(u)>0$ and rearranging gives
    \[
        \lambda_1+\ell\left(1-\frac{1}{\lambda_1}\right)\leq \deg(u).
    \]
\end{proof}

\begin{lemma}\label{lemma:compare}
    Let $T$ be a tree and $u\neq v$ vertices. Let $T_u$ and $T_v$ be trees obtained from $T$ by attaching a leaf to $u$ and $v$.
    If $\phi(T-u,t)>\phi(T-v,t)$ for $t>\lambda_1(T)$, then $\lambda_1(T_u)>\lambda_1(T_v)$.
\end{lemma}
\begin{proof}
    First, we prove that
    \[
        \phi(T_u,t)<\phi(T_v,t)
    \]
    for any $t\ge \lambda_1(T)$. This is sufficient, since it implies that $\phi(T_v,\lambda_1(T_u))>0$ and by interlacing we know that $\phi(T_v,t)$ has exactly one change of sign in $(\lambda_1(T),\infty)$ and $\phi(T_v,t)$ is positive as $t\to\infty$. This means that $\lambda_1(T_u)>\lambda_1(T_v)$.

    We have
    \begin{align*}
        \phi(T_u,t)&=t\phi(T,t)-\phi(T-u,t),\\
        \phi(T_v,t)&=t\phi(T,t)-\phi(T-v,t).
    \end{align*}
    Therefore we have
    \[
        \phi(T_v,t)-\phi(T_u,t)=\phi(T-u,t)-\phi(T-v,t)>0.\qedhere
    \]
\end{proof}

\begin{proposition}
    The only sustainable trees with $\Delta=\frac{n}{2}$ are $P_4, D_6,$ the family $F_1(n)$ for $n\geq 8$ vertices and the tree $X$ obtained from $S_3$ and $P_4$ by connecting the degree 3 vertex and one of the degree 2 vertices.
\end{proposition}
\begin{proof}
For $n\leq 12$, the statement was verified by exhaustive enumeration of all trees on $n$ vertices. The enumeration code and output are included in the ancillary files. The sustainable trees found in this range are exactly
\[
    P_4,\quad D_6,\quad X,\quad \text{and }F_1\text{ for }8\leq n\leq 12.
\]
It remains to consider $n\geq 14$.

Any tree not in $\mathcal{F}$ with $n\geq 14$ is non-sustainable by \cref{thm: degree n/2 non-sustainable}. This restricts the set of remaining trees whose sustainability needs to be verified. Note that, since $\Delta=n/2$ for all of these trees in $\mathcal{F}$, we know that the vertex $u$ (as depicted in \cref{fig: exceptional family}) is always the principal vertex and therefore any edge not connected to $u$ is sustainable by \cref{lem:ERLsustainable}.

\textbf{Claim: $F_1(n)$ is sustainable.} Any leaf $w$ connected to $u$ can be replaced by the edge $vw$ to get an isomorphic copy of the original tree, thus all the leaf edges are sustainable. Now let $w$ be the degree 2 vertex connected to $u$. By \cref{lem: incident degree 2 is sustainable}, this edge is sustainable and therefore $F_1(n)$ is sustainable.

\textbf{Claim: $F_3(n)$ is non-sustainable.} This was proven in \cref{ex: double star non-sustainable}.

\textbf{Claim: $F_4(n)$ and $F_5(n)$ are non-sustainable.}
Both of these trees can be obtained from a double star graph $S_{n/2-1,n/2-2}$ by attaching a pendant vertex to one of its leaves. Since $n\geq14$, we have $\frac n2-1\geq3$ and $\frac n2-2\geq3$, so \cref{prop: double star with pendant is non-sustainable} applies. Hence $F_4(n)$ and $F_5(n)$ are non-sustainable.

\textbf{Claim: $F_6(n)$ and $F_8(n)$ are non-sustainable.}
In both cases we show that the edge $uv$ is non-sustainable. Delete $uv$, and first consider a replacement edge $uw$, where $w$ lies in the component containing $v$. Let
$P$ be the path between $v$ and $w$ path in that component. In both $F_6(n)$ and $F_8(n)$ this path lies on one of the pendent arms at $v$.
Let $T'=T-uv+uw$. Clearly $\lambda_1(T')>2$ and $P$ is not a pendent path as it is between $u$ and $v$ of degrees at least 3. Hence by \cref{lem: contracting plus adding leaf}iteratively contracting edges of $P$ and adding leaf to $v$ or the last produced leaf vertex we obtain the tree $T$. Hence $\lambda_1(T)>\lambda_1(T')>\lambda_1(T)$, that is a contradiction.

Finally, after deleting $uv$, every vertex in the component containing $u$ other than $u$ is a leaf adjacent to $u$. If a replacement joins such a leaf $x$ to a vertex $w$ in the other component, a generalized tree shift with beneficiary $u$ transforms $T-uv+xw$ into $T-uv+uw$ and strictly increases the spectral radius. Hence this replacement also has spectral radius strictly smaller than $\lambda_1(T)$. Therefore every replacement of $uv$ decreases the spectral radius, so $F_6(n)$ and $F_8(n)$ are non-sustainable.

\textbf{Claim: $F_2(n)$ is non-sustainable.}
Let $T=F_2(n)$ and let us assume that it is sustainable, in particular every leaf edge at $u$ is sustainable. Let $T'$ be a tree that can be obtained from $T$ by deleting a leaf edge at $u$ and adding a leaf edge to a vertex $w$, such that $\lambda_1(T')\ge \lambda_1(T)$. Without loss of generality we may assume that $w$ maximizes the Perron vector of $T'$ and $w\neq u$.
Since, $\Delta(T')\ge n/2-1$, therefore $\lambda_1(T')>\sqrt{n/2-1}\ge \sqrt{6}$. As the number of leaf edges at $w$ is at least 1, therefore by \cref{lemma:cheap_trick} we have that the  degree of the principal vertex of $T'$ has to be at least
\[
     1+\sqrt{6}-\tfrac{1}{\sqrt{6}}>3,
\]
in particular, it shows that $w=v$, where $v$ is as defined in \cref{fig: exceptional family}.
Now we apply \cref{lemma:compare} on $\widehat T$ obtained from $T$ by deleting a leaf at $u$ and compute the corresponding characteristic polynomials.
\begin{align*}
    \phi(\widehat T-u,t)&=t\phi(\widehat T-u,v,t)-(\tfrac{n}{2}-3)t^{n/2-2+n/2-4}(t^2-1)-t^{n/2-2+n/2-2}\\
    \phi(\widehat T-v,t)&=t\phi(\widehat T-v,u,t)-(\tfrac{n}{2}-1)t^{n/2-2+n/2-4}(t^2-1)
\end{align*}
Therefore,
\[
\phi(\widehat T-u,t)-\phi(\widehat T-v,t)=2t^{n-6}(t^2-1)-t^{n-4}=t^{n-6}(t^2-2)>0
\]
for all $t>\lambda_1(\widehat T)$; indeed $\widehat T$ contains the star $K_{1,n/2-1}$, so $\lambda_1(\widehat T)\ge\sqrt{n/2-1}\ge\sqrt{6}>\sqrt{2}$ because $n\ge 14$. Since $T=\widehat T_u$ and $T'=\widehat T_v$, \cref{lemma:compare} gives $\lambda_1(T)>\lambda_1(T')$, thus contradicting our assumption that $T=F_2(n)$ is sustainable.

\textbf{Claim: $F_7(n)$ is non-sustainable.} As before, for $T=F_7(n)$ we can argue similarly. That is assume that $T$ is sustainable, and we can reconnect a leaf from $u$ to some other vertex $w$, such that it has larger spectral radius and $w$ maximizes the Perron vector. By \cref{lemma:cheap_trick} we can conclude that $w=v$, since otherwise either
\[
 2=d_w\ge 1+\sqrt{6}-\tfrac{1}{\sqrt{6}}>3
\]
or
\[
 4=d_w\ge 3+\sqrt{6}-\tfrac{3}{\sqrt{6}}>4
\]
and both are contradictions. Now let $\widehat T$ be obtained from $T$ by deleting a leaf edge. Then
\begin{align*}
    \phi(\widehat T-u,t)&=t\phi(\widehat T-u,v,t)-(n/2-4)t^{n/2-2+n/2-5}(t^3-2t)-t^{n/2-2+n/2-2}\\
    \phi(\widehat T-v,t)&=t\phi(\widehat T-v,u,t)-(n/2-2)t^{n/2-3+n/2-4}(t^3-2t)
\end{align*}
and
\[
\phi(\widehat T-u,t)-\phi(\widehat T-v,t)=2t^{n-7}(t^3-2t)-t^{n-4}=t^{n-6}(t^2-4)>0
\]
for all $t>\lambda_1(\widehat T)$; indeed $\widehat T$ contains the star
  $K_{1,n/2-1}$, so $\lambda_1(\widehat T)\ge\sqrt{n/2-1}\ge\sqrt{6}>2$ because
  $n\ge 14$. Since $T=\widehat T_u$ and $T'=\widehat T_v$, \cref{lemma:compare}
  gives $\lambda_1(T)>\lambda_1(T')$, thus contradicting our assumption that
  $T=F_7(n)$ is sustainable.
\end{proof}

The main technicality of the previous theorem is due to the lack of stability of the principal vertices under reconnecting leaf edges. Specifically, for trees in $\mathcal{F}$ the principal vertices could be different vertices after we reconnect a non-sustainable leaf edge, while for trees not from $\mathcal{F}$ with $\Delta(T)=n/2$ the principal vertex remains principal as in the proof of Theorem~\ref{thm: large degree is non-sustainable}.

\section{Future directions and open problems}\label{sec:future}

The results in this paper suggest that sustainability is a spectral measure of how robust a tree is under local edge replacement. At one extreme, paths are sustainable: every edge can be replaced without decreasing the spectral radius. At the other extreme, stars are non-sustainable: replacing any leaf edge necessarily decreases the spectral radius. Our results show that this star-like behaviour persists for all trees with maximum degree at least $n/2+1$.

\begin{figure}[htb]
    \centering
    \begin{tikzpicture}[main_node/.style={circle,draw,minimum size=0.8em,inner sep=3pt}]
    \begin{scope}
        \node[] at (0, 1) {$T_1$};
        \node[main_node] (0) at (0,0) {};
        \node[main_node] (1) at (1,0) {};
        \node[above, yshift=2pt, xshift=7pt] at (1) {$v$};
        \node[main_node] (2) at (2,0) {};
        \node[above, yshift=2pt, xshift=7pt] at (2) {$u$};
        \node[main_node] (3) at (3,0) {};
        \node[main_node] (4) at (4,0) {};
        \node[main_node] (5) at (5,0) {};
        \node[main_node] (6) at (6,0) {};
        \node[main_node] (7) at (3,1) {};

         \path[draw, thick]
        (7) edge node {} (3)
        (3) edge node {} (2)
        (2) edge node {} (1)
        (1) edge node {} (0)
        (3) edge node {} (4)
        (4) edge node {} (5)
        (5) edge node {} (6)
        ;
        \draw[->] (6.5,0) -- (8, 0 ) node[midway, sloped, yshift=5pt] {\footnotesize GTS};
    \end{scope}
    \begin{scope}[xshift=8.5cm]
        \node[] at (0, 1) {$T_2$};
        \node[main_node] (0) at (0,0) {};
        \node[main_node] (1) at (1,0) {};
        \node[above, yshift=2pt, xshift=7pt] at (1) {$v$};
        \node[main_node] (2) at (2,0) {};
        \node[above, yshift=2pt, xshift=7pt] at (2) {$u$};
        \node[main_node] (3) at (3,0) {};
        \node[main_node] (4) at (4,0) {};
        \node[main_node] (5) at (5,0) {};
        \node[main_node] (6) at (6,0) {};
        \node[main_node] (7) at (3,1) {};

         \path[draw, thick]
        (7) edge node {} (3)
        (3) edge node {} (2)
        (2) edge node {} (1)
        (3) edge node {} (4)
        (4) edge node {} (5)
        (5) edge node {} (6)
        ;
        \draw[thick] (0) to [bend left=50] (2);
    \end{scope}
    \end{tikzpicture}
    \caption{The tree $T_1$ is non-sustainable by \cref{prop: starlike trees}, whereas $T_2$ is sustainable, which can be verified by computation. The tree $T_2$ can be obtained by a generalized tree shift from $T_1$ with $u$ and $v$ being the beneficiary and candidate, respectively.}
    \label{fig:Non upperset}
\end{figure}
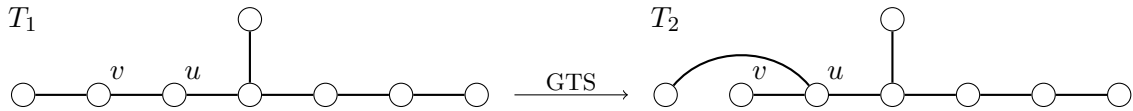

One natural direction is to compare sustainability with Csikvári's partial order on trees \cite{Csikvari2010}, generated by generalized tree shifts. This order has the path as its minimum element and the star as its maximum element. Since every nontrivial generalized tree shift strictly increases the spectral radius, it is natural to ask whether sustainability is monotone in this order. The examples in \cref{fig:Non upperset} show that it is not: a generalized tree shift may transform a non-sustainable tree into a sustainable one. Thus the sustainable trees do not form an order ideal, and the non-sustainable trees do not form an order filter.  We pose the following weaker connectivity question.

\begin{openproblem}\label{op:GTS-connectivity}
    Let $X$ be a sustainable tree on $n$ vertices. Does there exist a sequence of trees
    \[
        X_1,X_2,\dots,X_k
    \]
    such that $X_1\cong P_n$, $X_k\cong X$, each $X_{i+1}$ is obtained from $X_i$ by a single generalized tree shift, and every $X_i$ is sustainable?

    Analogously, let $X$ be a non-sustainable tree on $n$ vertices. Does there exist a sequence of trees
    \[
        X_1,X_2,\dots,X_k
    \]
    such that $X_1\cong X$, $X_k\cong K_{1,n-1}$, each $X_{i+1}$ is obtained from $X_i$ by a single generalized tree shift, and every $X_i$ is non-sustainable?
\end{openproblem}
Equivalently, this problem asks whether the sustainable trees form a connected subposet containing $P_n$, and whether the non-sustainable trees form a connected subposet containing $K_{1,n-1}$, when one restricts the tree-shift order to each class.

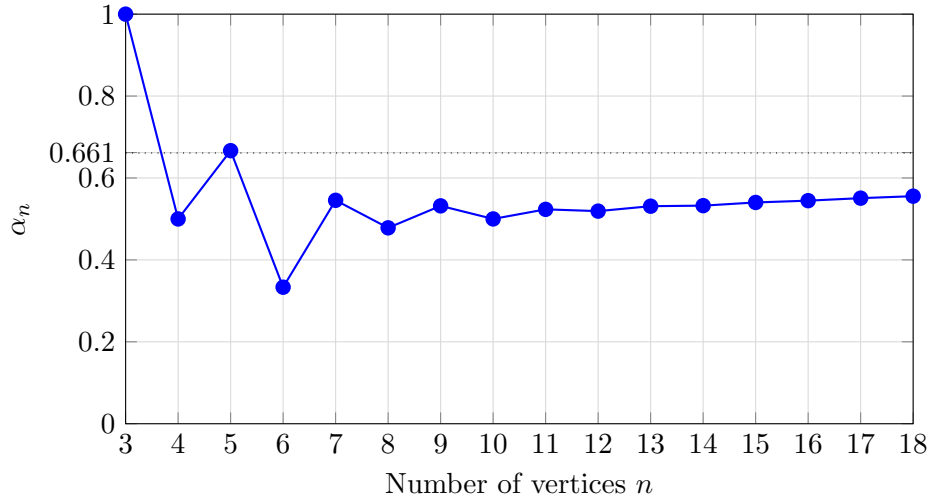
\begin{figure}[htb]
    \centering
    \begin{tikzpicture}
    \begin{axis}[
    width=12cm,
    height=7cm,
    xlabel={Number of vertices $n$},
    ylabel={$\alpha_n$},
    xmin=3, xmax=18,
    ymin=0, ymax=1,
    xtick={3,...,18},
    ytick={0,0.2,0.4,0.6,0.662,0.8,1},
    yticklabels={0,0.2,0.4,0.6,0.661,0.8,1},
    grid=both,
    major grid style={gray!30},
    minor grid style={gray!10},
    mark size=2.5pt
    ]

    \addplot[
        blue,
        thick,
        mark=*,
    ] coordinates {
        (3,1)
        (4,0.5)
        (5,0.6666666667)
        (6,0.3333333333)
        (7,0.5454545455)
        (8,0.4782608696)
        (9,0.5319148936)
        (10,0.5)
        (11,0.5234042553)
        (12,0.5190562613)
        (13,0.5311299001)
        (14,0.5324469769)
        (15,0.540240279)
        (16,0.5446169772)
        (17,0.5507207633)
        (18,0.5555717019)
    };

    \addplot[
        dotted,
        domain=3:18,
        samples=2
    ] {0.661};

    \end{axis}
    \end{tikzpicture}
    \caption{Graph of ratios $\alpha_n$ of sustainable trees on $n$ vertices, where $3\le n\le 18$ and the conjectured limit $1-\rho\approx 0.661$.}
    \label{fig:ratio of sustainable}
\end{figure}

 Another question we could ask, how likely a randomly chosen tree to be sustainable? To be precise, let $s_n$ denote the number of sustainable trees on $n$ vertices, counted up to isomorphism, and let $t_n$ denote the total number of trees on $n$ vertices, counted up to isomorphism. Define
\[
\alpha_n=\frac{s_n}{t_n}.
\]
The computations in \cref{fig:ratio of sustainable} suggest that $\alpha_n$ may approach a nontrivial limit. The following upper bound can already be proved.

\begin{proposition}\label{prop:positive-density-nonsustainable}
Let $\rho\approx0.338$ be the radius of convergence appearing in Otter's asymptotic formula \cite{otter1948number} for unlabelled trees. Then
\[
    \limsup_{n\to\infty}\alpha_n\leq 1-\rho^2\approx0.885.
\]
\end{proposition}

\begin{proof}
For $n\ge 4$ let $T$ be a tree on $n-2$ vertices. For $u\in V(T)$, let $T_u^{(2)}$ be obtained from $T$ by attaching two new leaves at $u$, and choose $v\in V(T)$ so that
\[
    \lambda_1\bigl(T_v^{(2)}\bigr)
    =\max_{u\in V(T)}\lambda_1\bigl(T_u^{(2)}\bigr).
\]
Write $H=T_v^{(2)}$, and let $x$ and $y$ be the two new leaves at $v$. We claim that the edge $vx$ is not sustainable.

Suppose otherwise. Then for some $w\neq v$ the tree
\[
    H'=H-vx+xw
\]
satisfies $\lambda_1(H')\geq\lambda_1(H)$.

If $w=y$, then $H$ is obtained from $H'$ by the generalized tree shift with beneficiary $v$ and candidate $y$, moving the edge $yx$ to $vx$. Since $n\geq4$, this shift is proper, and hence \[ \lambda_1(H)>\lambda_1(H'), \] a contradiction. Thus we may assume that $w\neq y$.

Let $\Zv'$ be the Perron vector of $H'$. If $\Zv'(v)\geq\Zv'(w)$, rotating the edge $xw$ back to $vx$ strictly increases the spectral radius, contradicting $\lambda_1(H')\geq\lambda_1(H)$. If $\Zv'(w)>\Zv'(v)$, rotating the other new leaf edge $vy$ to $wy$ produces $T_w^{(2)}$ and strictly increases the spectral radius and we have the chain of inequalities, \[\lambda_1\bigl(T_w^{(2)}\bigr)>\lambda_1(H')\geq\lambda_1(H)=\lambda_1\bigl(T_v^{(2)}\bigr).\] This contradicts the choice of $v$. Hence $vx$ is non-sustainable.

Moreover, $v$ is the unique principal vertex of $H$. Indeed, if a principal vertex were not incident with $vx$, then \cref{lem:ERLsustainable} would imply that $vx$ is sustainable. If $\Zv$ is the Perron vector of $H$, then the leaf $x$ cannot be principal because
\[
    \Zv(x)=\frac{\Zv(v)}{\lambda_1(H)}<\Zv(v).
\]
Consequently, the isomorphism class of $T$ can be recovered from $H$ by locating its unique principal vertex and deleting two of its leaf neighbours. The map $T\mapsto H$ is therefore injective on isomorphism classes, and
\[
    t_n-s_n\geq t_{n-2}.
\]
It follows that
\[
    \alpha_n\leq 1-\frac{t_{n-2}}{t_n}.
\]
Otter's asymptotic formula gives
\[
    t_n\sim c\,n^{-5/2}\rho^{-n}
\]
for a positive constant $c>0$ \cite{otter1948number}.
Hence $t_{n-2}/t_n\to\rho^2$, proving the result.
\end{proof}

\begin{openproblem}\label{op:asymptotic-density}
Does the sequence $(\alpha_n)$ converge and has limit $1-\rho\approx 0.661$? More basically, is
\[
    \liminf_{n\to\infty}\alpha_n>0?
\]
\end{openproblem}

Note that, several arguments in this paper depend on understanding where a leaf should be attached in order to maximize the spectral radius. This suggests the following more general problem.
Let $(Y,y)$ be a rooted tree. For a tree $T$ and a vertex $u\in V(T)$, let $T_u(Y,y)$ be the tree obtained by  identifying $y$ with $u$. One may ask which vertices $u$ maximize
\[
    \lambda_1\bigl(T_u(Y,y)\bigr)
\]
and we call such a vertex a \emph{$(Y,y)$-center} of $T$.
\begin{openproblem}\label{op:rooted attachment centers}
Given a rooted tree $(Y,y)$, give necessary and sufficient conditions for a vertex of a tree $T$ to be the $(Y,y)$-centers of $T$. In particular, how do these centers depend on the Perron vector of $T$ and on the structure of the rooted tree $(Y,y)$?
\end{openproblem}

When $Y$ is a single rooted edge, this asks where one should attach a new leaf in order to maximize the spectral radius. This special case is closely related to the Perron vector of $T$, and it appears naturally when testing whether a pendent edge is sustainable. For larger rooted trees $(Y,y)$, the answer  depends not only on the Perron vector entries of $T$, but also on walk-generating functions rooted at the candidate attachment vertices.

\subsection*{Acknowledgements and AI Usage}

The notion of sustainability, as well as several of the initial results and techniques used here, were introduced in Maarten van Dijk's Master's thesis \cite{VanDijk2025SustainableTrees}.

\paragraph{Use of generative AI.} Most of the research and the initial writing of this manuscript was done before the widespread availability of highly capable LLMs. Generative AI tools were used only at a later stage for proofreading, as well as for assistance with minor \LaTeX{} errors. They were not used to generate mathematical ideas, proofs, results, or any substantive amounts of text.

\providecommand{\doi}[1]{\href{https://doi.org/#1}{\nolinkurl{doi:#1}}}\providecommand{\arxiv}[1]{\href{https://arxiv.org/abs/#1}{#1}}

\end{document}